\documentclass[11pt]{article}

\usepackage{listings}
\usepackage[colorlinks=false, backref]{hyperref}

\usepackage{amsmath,amsthm,amsfonts,amssymb,amscd}
\usepackage{float}
\usepackage{graphicx}

\usepackage[mathscr]{euscript}
\usepackage{stmaryrd}
\usepackage{microtype}
\usepackage{booktabs}
\usepackage{cleveref}
\usepackage{bookmark}
\usepackage{mathrsfs}

\usepackage{emptypage}
\usepackage{tikz}
\usepackage{subfigure}
\usepackage[utf8]{inputenc}
\usepackage[T1]{fontenc}

\DeclareMathAlphabet{\mathpzc}{OT1}{pzc}{m}{it}

\theoremstyle{plain}
\newtheorem{theorem}{\scshape Theorem}
\newtheorem{proposition}[theorem]{\scshape Proposition}

\newtheorem{corollary}[theorem]{\scshape Corollary}

\theoremstyle{definition}

\newtheorem{remark}{\scshape Remark}

\theoremstyle{definition}

\def\bbR{{\mathbb R}}

\def\bbT{{\mathbb T}}

\def\p{\text{\bf\emph{p}}}

\definecolor{grey}{rgb}{0.5,0.5,0.5}
\definecolor{lightgrey}{rgb}{0.9,0.9,0.9}
\definecolor{darkgreen}{rgb}{0,0.6,0}
\definecolor{orange}{rgb}{1,0.5,0}

\def\p{{\partial\hspace{1pt}}}

\def\({{(\hspace{-2pt}(}}
\def\){{)\hspace{-2pt})}}

\def\jump#1{{[\hspace{-2pt}[#1]\hspace{-2pt}]}}

\def\XXint#1#2#3{{\setbox0=\hbox{$#1{#2#3}{\int}$}
\vcenter{\hbox{$#2#3$}}\kern-.5\wd0}}

\usepackage{fancyhdr}
\title{A model for Rayleigh-Taylor mixing and interface turn-over}

\author{Rafael Granero-Belinch\'on
\\Department of Mathematics
\\University of California
\\Davis, CA 95616 USA
\\{\footnotesize email: rgranero@math.ucdavis.edu}
\and
 Steve Shkoller
\\Department of Mathematics
\\University of California
\\Davis, CA 95616 USA
\\{\footnotesize email: shkoller@math.ucdavis.edu}
}
\date{July 5, 2015}
\begin{document}

%\begin{abstract}
%\end{abstract}

\maketitle

\begin{center} {\bf Abstract}
\end{center}

We first develop a new mathematical model for two-fluid interface motion, subjected to the  Rayleigh-Taylor (RT) instability in two-dimensional fluid flow, which in its simplest form, is given by
$ h_{tt}( \alpha ,t)  = A g\, \Lambda h  -  \frac{ \sigma }{ \rho^++\rho^-} \Lambda^3 h
 - A \p_\alpha(H h_t h_t) $, where $\Lambda = H \p_ \alpha $ and $H$ denotes the Hilbert transform.   In this so-called  $h$-model,
 $A$ is the Atwood number, $g$ is the acceleration,  $ \sigma $ is surface tension, and $\rho^\pm$ denotes the densities of the two fluids.
  We derive our $h$-model using asymptotic expansions in the
Birkhoff-Rott
integral-kernel formulation for the evolution of an interface separating two incompressible and irrotational fluids.  The resulting  $h$-model equation is shown to be locally and globally
well-posed in Sobolev spaces when a certain {\it stability condition} is satisfied; this stability condition requires that the  product of the Atwood number and the 
initial velocity field  be positive.   The asymptotic  behavior of these global solutions, as $t \to \infty $,  is also described.
The $h$-model equation is shown to have  interesting balance laws, which distinguish the stable dynamics from the  unstable dynamics.
Numerical simulations of the  $h$-model  show that the interface can quickly grow due to nonlinearity, and then stabilize when the lighter fluid is 
on top 
of the heavier fluid and acceleration is directed downward.   In the unstable case of a heavier fluid being supported by the lighter fluid, we find good agreement for the growth of the mixing
layer with experimental data in the ``rocket rig'' experiment of Read of Youngs.

We then derive an interface model for RT instability,  with a general  parameterization $z( \alpha ,t) = \left( z_1(\alpha ,t), z_2 (\alpha ,t)\right)$
such that $z$ satisfies
$
z_{tt}= \Lambda\bigg{[}\frac{A}{|\partial_\alpha z|^2}H\left(z_t\cdot (\partial_\alpha z)^\perp H(z_t\cdot (\partial_\alpha z)^\perp)\right)
 + \frac{\jump{p}}{\rho^++\rho^-}  + A g z_2 \bigg{]} \frac{(\partial_\alpha z)^\perp}{|\partial_\alpha z|^2} 
  +z_t\cdot (\partial_\alpha z)^\perp\left(\frac{(\partial_\alpha z_t)^\perp}{|\partial_\alpha z|^2}-\frac{(\partial_\alpha z)^\perp 2(\partial_\alpha z\cdot \partial_\alpha z_t)}{|\partial_\alpha z|^4}\right)$.
 This more general RT $z$-model allows for interface turn-over.  Numerical simulations  of the $z$-model show an even better agreement with  the predicted mixing layer growth for the ``rocket rig'' experiment.

 \vspace{.3 in}
 
 \tableofcontents
 
\section{Introduction}
The instability of a heavy fluid layer supported by a light one is generally known as Rayleigh-Taylor
(RT) instability (see Rayleigh \cite{Ra1878} and Taylor \cite{Ta1950}).  It can occur under gravity and, equivalently, under an acceleration of the fluid system
in the direction toward the denser fluid; in particular,  RT is an interface fingering instability, which occurs when a perturbed
interface, between two fluids of different
density, is subjected to a normal pressure gradient. Whenever the pressure is higher in the lighter
fluid, the differential acceleration causes the two fluids to mix.  See Sharp \cite{Sharp1984}, Youngs \cite{Youngs1984,Youngs1989}, and Kull \cite{Kull1991} for
an overview of the RT instability.

The Euler equations of inviscid  hydrodynamics serve as the basic mathematical model for  RT  instability and mixing between two fluids.   This highly unstable 
system of conservation laws is both difficult to analyze (as it is ill-posed in the absence of surface tension and viscosity) and difficult to computationally simulate at the small spatial scales of RT mixing.   As such,
our objective is to develop  {\it model equations}, which can be used to predict the RT mixing layer and growth rate.

In order to derive our RT interface models, we shall assume both incompressible and irrotational flow for the two-fluid Euler equations.   Rather than proceeding with a direct 
approximation of the Euler equations, we shall instead work with the equivalent Birkhoff-Rott singular integral-kernel formulation for the evolution of the material interface.  We have found
that this formulation possesses  a certain robustness in regards to approximations founded upon expansions in various parameters.

In the simplest case, in which the interface is modeled as a graph $( \alpha , h( \alpha , t))$ of a signed {\it height} function $h( \alpha , t)$, 
Our approach yields a second-order in-time quadratically nonlinear wave equation for the position of the interface, the $h$-model, which is given
by
$$ h_{tt}( \alpha ,t)  = A g\, \Lambda h  -  \frac{ \sigma }{ \rho^++\rho^-} \Lambda^3 h
 - A \p_\alpha(H h_t h_t) \,,$$
  where $\Lambda = H \p_ \alpha $ and $H$ denotes the Hilbert transform.   In this $h$-model equation,
 $A$ denotes the Atwood number, $g$ is the acceleration,  $ \sigma \ge 0$ is the surface tension, and $\rho^\pm>0$ denotes the densities of the two fluids.

As we will describe below, our $h$-model equation for RT instability is  both locally and globally  well-posed in Sobolev spaces when  a stability condition is satisfied, which
requires the  product of the Atwood number and the initial velocity field to be  positive.   Under such a stability condition, we also derive the
asymptotic behavior of  solutions to the $h$-model as $t \to \infty $.
A number of interesting energy laws are also derived, which distinguish between stable and unstable interface dynamics

Numerical simulations are performed, which show that the $h$-model is capable of producing remarkable growth of the interface, followed by 
(possibly oscillatory) decay to a rest
state in the case that the lighter fluid is supported by the heavier fluid. 
In the highly unstable case, where a heavy fluid is supported by the lighter fluid, 
 we perform the classical ``rocket rig'' experiment of Read \cite{Read1984} and  Youngs \cite{Youngs1989} for the case of
unstable RT mixing-layer growth, and find very good agreement for the growth rates with the predicted quadratic growth rate of
Youngs \cite{Youngs1989}, and with both Direct Numerical Simulations
and experimental data.  

Finally, in order to allow for the interface to turn-over (rather than simply remaining a graph) we develop a more general $z$-model for the interface 
parameterization $z( \alpha ,t) = \left( z_1(\alpha ,t), z_2(\alpha ,t)\right)$ which satisfies
\begin{align*} 
z_{tt}&= \Lambda\bigg{[}\frac{A}{|\partial_\alpha z|^2}H\left(z_t\cdot (\partial_\alpha z)^\perp H(z_t\cdot (\partial_\alpha z)^\perp)\right)
 + \frac{\jump{p}}{\rho^++\rho^-}  + A g z_2 \bigg{]} \frac{(\partial_\alpha z)^\perp}{|\partial_\alpha z|^2} \\
&\qquad \qquad
  +z_t\cdot (\partial_\alpha z)^\perp\left(\frac{(\partial_\alpha z_t)^\perp}{|\partial_\alpha z|^2}-\frac{(\partial_\alpha z)^\perp 2(\partial_\alpha z\cdot \partial_\alpha z_t)}{|\partial_\alpha z|^4}\right) \,.
\end{align*} 
Rather than constraining the interface amplitude to grow at the RT instability, the $z$-model can allow for interface turn-over.  We perform
numerical simulations to demonstrate the improvement afforded by this more general model, and the even better accuracy in matching the 
predicted growth of the RT mixing layer for the ``rocket rig'' experiment.  We also perform the so-called ``tilted rig'' experiment, in which the tank
holding the fluid is titled by a small angle away from vertical.  Again, our simulations demonstrate good qualitative agreement with DNS.
To conclude, we also perform a numerical simulation for a  Kelvin-Helmholtz problem,  for which the Atwood number is set to zero to prevent 
an RT 
instability; starting from a steep mode-$1$ wave, we show  roll-over of the interface with a very localized energy spectrum.

%The RT instability  evolves from
%a multiple wavelength initial perturbation of a ``flat'' interface, and subsequent turbulent
%mixing occurs as shown in Youngs \cite{Youngs1984}.   This mixing is fundamental to the understanding of many natural and technological phenomena.   The
%Euler equations of fluid dynamics are the mathematical 

\section{Equations for multi-phase flow} The Euler equations are a system of conservation laws, modeling the dynamics of multi-phase inviscid fluid flow.   
For incompressible two-dimensional motion, 
the conservation of momentum for an homogeneous, inviscid fluid can be written as
\begin{equation}\label{eqa1}
\rho(u_t+(u\cdot \nabla)u)+\nabla p=-g\rho e_2,
\end{equation}
where $u=(u_1,u_2)$  denotes the velocity vector field of the fluid, 
$p$ is the scalar pressure function, $\rho$ is the density, $g$ is the acceleration,  $\nabla=(\partial_{x_1},\partial_{x_2})$ is the gradient vector, and
$e_2=(0,1)$.

For incompressible flow, the conservation of mass is given by
\begin{equation}\nonumber
\nabla\cdot u=0\,.
\end{equation}
We shall further assume that the fluid is irrotational so that
\begin{equation}\label{eqa3}
\omega:= \operatorname{curl} u=0,
\end{equation}
where $\omega$ is the vorticity of the fluid, and $ \operatorname{curl}  u = \partial_1 u^2 - \partial_2 u^1$.

We assume that there are two fluids with densities  $\rho^+$ and $\rho^-$, separated by a material interface $\Gamma(t)$, where $t$ denotes an instant of time 
in the interval $[0,T]$.  As shown in Figure \ref{fig0}, the fluid with density $\rho^+$ lies above $\Gamma(t)$, while the fluid with density $\rho^-$ lies below  $\Gamma(t)$.
\begin{figure}[h]
\begin{center}
\begin{tikzpicture}[scale=0.5]
    \draw (21,1.) node { $\Gamma(t)$}; 
    \draw (17.2,2.) node { $u^+, \rho^+$}; 
    \draw (15.2,0.) node { $u^-,\rho^-$}; 
    \draw[color=blue,ultra thick] plot[smooth,tension=.6] coordinates{( 10,0) (11,2) (13,-1) (15, 3) (17, 0) (19, 2) (21,-1) };
\end{tikzpicture} 
\end{center}
\vspace{-.3 in}
\caption{{\footnotesize The blue curve is an illustration of the interface $\Gamma(t)$,  separating two fluids at a time $t \in [0,T]$.   The fluid on top of  $\Gamma(t)$ has density $\rho^+$, while the fluid on the bottom has density $\rho^-$.}}
 \label{fig0}
\end{figure}
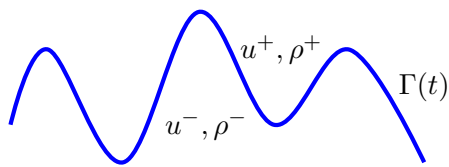

For each instant of time $t\in[0,T]$, the 
interface $\Gamma(t)$ is parameterized by a function $z( \alpha , t)$.   We will provide a special choice for the parameterization $z( \alpha ,t)$ below.   

We denote the jump of a field variable $f(x,t)$ across $\Gamma(t)$ by
$$\jump{f}=f^+-f^-\,.$$
We let  $n$  and $ \tau$ denote the unit normal and tangent vectors to $\Gamma(t)$, respectively; for the RT instability, we have the following jump conditions:
\begin{equation}\label{jump_cond}
\jump{u\cdot n}=0\,, \ \ \ \jump{u\cdot \tau}\neq 0 \,, \ \ \text{ on } \Gamma(t) \,.
\end{equation} 
Note, that  $\partial_\alpha z( \alpha ,t)$ is a tangent vector to $\Gamma(t)$ at the point $z( \alpha ,t)$, so that
$$
\jump{u\cdot \partial_\alpha z}\neq 0.
$$ 
Consequently, the velocity is not continuous at the interface. Due to the fact that the {\it shape} of the interface is determined only by the normal  component of the
fluid velocity, the motion of the interface has a tangential reparameterization symmetry, so that we can 
add any tangential velocity to the motion of the interface, with the hope that this will simplify the analysis.   As such, 
 the evolution equation for the parameterization of $\Gamma(t)$ is written as
\begin{equation}\label{eqb}
z_t( \alpha  ,t) =u(z( \alpha ,t),t)+c( \alpha ,t) \partial_\alpha z(\alpha ,t),
\end{equation}
where the function $c(\alpha ,t)$ will be specified below.

\section{The integral kernel formulation}
A number of modal models have been proposed for the evolution of the RT mixing layer; see, for example, Rollin \& Andrews \cite{RoAn2013} and
Goncharov \cite{Goncharov2002}, and the references therein.  These modal models are based on a modal decomposition and approximation of the
evolution equations for the parameterization of the interface and the velocity potential; such models   consist of a large coupled system of nonlinear ODEs for a finite set of Fourier modes.   

Rather than developing a modal model for RT mixing and approximating the partial differential equations themselves, 
we shall take another approach to the development of an RT model,  which is founded upon the Birkhoff-Rott  integral-kernel formulation 
for the evolution of an interface, separating two incompressible and irrotational fluids.  

In order to introduce the integral-kernel formulation, which consists of {\it singular integrals}, we 
 begin by defining the {\it principal value integral } of a given function $f$ as
$$
\text{P.V.}\int_\bbR f (\beta) d \beta =\lim_{\epsilon\rightarrow0^+}\int_{(-1/\epsilon,-\epsilon)\cup(\epsilon,1/\epsilon)} f (\beta) d  \beta \,.
$$
The well known Biot-Savart kernel $\mathcal{K}_{\small \operatorname{BS}}$ is an integral representation for  $\nabla^\perp\Delta^{-1}$  where 
$\nabla^\perp=(-\partial_{x_2},\partial_{x_1})$, and $ \Delta = \partial_{x_1}^2 + \partial_{x_2} ^2$ denotes the Laplace operator in the plane.
 In other words, if the two fluids fill the plane, then $\Delta^{-1}$ is given by the Newtonian potential and the 
Biot-Savart kernel is given by
\begin{equation}\label{eq2}
\mathcal{K}_{\small \operatorname{BS}}(x,y)=\frac{1}{2\pi}\nabla^\perp \log(x)=\frac{1}{2\pi}\left(-\frac{x_2-y_2}{(x_2-y_2)^2+(x_1-x_1)^2}, \frac{x_1-y_1}{(x_2-y_2)^2+(x_1-y_1)^2}\right).
\end{equation}
Similarly, if the fluid flow is periodic in the horizontal variable, then
\begin{equation}\label{eq2b}
\mathcal{K}_{\small \operatorname{BS}}(x,y)=\frac{1}{4\pi} \left(\frac{-\sinh(x_2-y_2)}{\cosh(x_2-y_2)-\cos(x_1-y_1)},\frac{\sin(x_1-y_1)}{\cosh(x_2-y_2)-\cos(x_1-y_1)}\right).
\end{equation}

Due to the characteristics of the irrotational flow, the {\it vorticity} is a measure which is supported on the interface $\Gamma(t)$,  written as
$$
\omega=\varpi \delta_{\Gamma(t)} \,,
$$
where $\delta_{\Gamma(t)} $   is the Dirac delta function supported on the interface $\Gamma(t)$. 
More precisely, the vorticity $\omega$ is a distribution defined as follows:  for all smooth test functions $\varphi$ with compact support,
$$
\omega(\varphi) = \int_ \mathbb{R}   \varpi(\beta,t)  \varphi( z( \beta , t)) d \beta \,.
$$
The function  $\varpi$ is the {\it amplitude of the vorticity} along $\Gamma(t)$.
Notice that $\varpi$ is minus the jump of the velocity in the tangential direction:
$$
\varpi=-\jump{u\cdot \partial_\alpha z}.
$$

Given the vorticity measure $\varpi$, we can reconstruct the velocity field everywhere; specifically, 
we have, thanks to the Biot-Savart law,  that
\begin{equation}\label{eq3}
u(x,t)=\text{P.V.}\int_\bbR \varpi(\beta)\mathcal{K}_{\small \operatorname{BS}}(x,z(\beta,t))d\beta.
\end{equation}

Since by \eqref{eqa3},  the velocity is irrotational, there exist a velocity potential function $\phi:\bbR^2\rightarrow\bbR$ such that $u^\pm=\nabla \phi^\pm$.    The potential $\phi$ satisfies Bernoulli's equation
\begin{equation}\label{bernoulli}
\phi_t + {\frac{1}{2}} | \nabla \phi|^2 + \frac{p}{\rho} + g x_2 = f(t) \,,
\end{equation} 
and $\phi$ is related to $\varpi$ by the singular integral equation
\begin{equation}\label{eq8}
\phi(x,t)={\frac{1}{2\pi}}\text{P.V.}\int_\bbR \varpi(\beta)\arctan\left(\frac{x_2-z_2(\beta,t)}{x_1-z_1(\beta,t)}\right)d\beta.
\end{equation}

As the kernel $\mathcal{K}_{\small \operatorname{BS}}(x,z(\beta,t))$ is singular for $x=z(\alpha,t)$, it follows
that the tangential component of  $u$ is discontinuous across the interface $\Gamma(t)$.  By similar reasoning, the potential function $\phi$ is 
also discontinuous across $\Gamma(t)$. {Moreover, the vorticity $\varpi$ along the interface $\Gamma(t)$ is related to the jump of the velocity potential by
\begin{equation}\label{ss1}
\varpi = - \partial _ \alpha \jump{\phi} \,.
\end{equation} 
}

Now, using the parameterization $z( \alpha , t)$ for the interface $\Gamma(t)$, we define the Birkhoff-Rott principal-value integral, denoting
$\int_ \mathbb{R}  $ by $ \int$, as
\begin{equation}
\label{eq1}
\mathcal{K}_{\small \operatorname{BR}}( \alpha,t )=\text{P.V.}\int \varpi(\beta,t) \mathcal{K}_{\small \operatorname{BS}}(z(\alpha),z(\beta,t))d\beta.
\end{equation}
In particular, \eqref{eq1} is explicitly given by
\begin{equation}
\label{eq1prime}
\mathcal{K}_{\small \operatorname{BR}}(\alpha,t )
=\frac{1}{2\pi}\text{P.V.}\int \varpi(\beta) \left(-\frac{z_2(\alpha,t)-z_2(\beta,t)}{|z(\alpha,t)-z(\beta,t)|^2}, \frac{z_1(\alpha,t)-z_1(\beta,t)}{|z(\alpha,t)-z(\beta,t)|^2}\right)d\beta.
\end{equation}
The Birkhoff-Rott function $\mathcal{K}_{\small \operatorname{BR}}(\alpha , t)$ denotes the average velocity at a given point  $\alpha \in \Gamma(t)$, so that
$$
\mathcal{K}_{\small \operatorname{BR}}=\frac{u^+ + u^-}{2}\text{ on the interface $z(\alpha,t)$}.
$$

\section{Dynamics of the interface and vorticity amplitude}

We  now formulate the dynamics of the interface.  Since $z_t = u + c \partial_ \alpha z$, we see that
\begin{equation}\label{z-dynamics}
z_t(  \alpha ,t) = {\frac{1}{2\pi}} {P.V.} \int \varpi( \beta) \frac{(z( \alpha ,t) - z(\beta,t))^\perp}{| z( \alpha ,t) - z(\beta,t)|^2} d\beta + c( \alpha , t) \partial_ \alpha 
z (\alpha , t) \,.
\end{equation}

In order to compute $\varpi_t$, we
use the identity \eqref{ss1}, and shall need to compute $\partial_t \jump{\phi}$.  To do so, we shall first compute 
compute $\jump{\phi}$ across the interface $\Gamma(t)$.      In fact, as we will explain below, it is easier to compute  
$\partial_ \alpha \jump{\phi}$,
and to do so, we study the limit of $\partial_ \alpha 
\phi^\pm(y,t)$ as the point $y$ approaches $\Gamma(t)$ in a direction normal to $\Gamma(t)$.    We recall that $ \partial_ \alpha z( \alpha ,t)$ is a 
tangent vector to $\Gamma(t)$ at the point $ z(\alpha ,t)$ and hence $ \partial_ \alpha^\perp z( \alpha ,t)  $ is a normal vector to
$\Gamma(t)$ at the point $ z(\alpha ,t)$.   For $0 < \epsilon \ll 1$,  and each time $t$,  we let
$$
y^\pm_ \epsilon( \alpha , t) := z( \alpha , t) \pm \epsilon  \partial_ \alpha^\perp z( \alpha ,t) 
$$
denote a sequence of points converging to $z( \alpha ,t)$ as $  \epsilon \to 0$ in the normal direction to $\Gamma(t)$.   Then,  the trace of
$\partial_ \alpha \phi^\pm$ on $\Gamma(t)$ is defined as the limit as $  \epsilon \to 0$ of the sequence $\partial_ \alpha \phi^\pm( y_ \epsilon ^\pm,t)$.

Since  $u^\pm = \nabla \phi^\pm$,  the chain-rule shows that
\begin{equation}\label{ss777}
\lim_{\epsilon\rightarrow0}\partial_\alpha \phi^\pm(y_ \epsilon ^\pm,t)
=\lim_{\epsilon\rightarrow0}u^\pm( y_ \epsilon ^\pm)\cdot\partial_\alpha y_ \epsilon ^\pm
=\lim_{\epsilon\rightarrow0}u^\pm( y_ \epsilon ^\pm)\cdot\partial_\alpha z \,.
\end{equation} 
In \cite{CoCoGa2010}, it was shown that
\begin{subequations}\label{u-lim}
\begin{align} 
\lim_{\epsilon\rightarrow0}u^+( y_ \epsilon ^+)\cdot\partial_\alpha z
& =\mathcal{K}_{\small \operatorname{BR}}( \cdot \partial_\alpha z - \frac{1}{2}\varpi \,, \\
\lim_{\epsilon\rightarrow0}u^-( y_ \epsilon ^-)\cdot\partial_\alpha z
& =\mathcal{K}_{\small \operatorname{BR}}\cdot \partial_\alpha z + \frac{1}{2}\varpi \,.
\end{align} 
\end{subequations}
Thus, from \eqref{ss777} and \eqref{u-lim}, it follows that
\begin{equation}\label{ss778}
\partial_ \alpha \jump{\phi} = \mathcal{K}_{\small \operatorname{BR}}\cdot \partial_\alpha z \,.
\end{equation} 
Next, we observe that $ \mathcal{K}_{\small \operatorname{BR}}\cdot \partial_\alpha z$ can be written as an exact derivative due to the
identity
\begin{equation}\label{ss779}
\frac{(z( \alpha ,t) - z(\beta,t))^\perp}{| z( \alpha ,t) - z(\beta,t)|^2}\cdot \partial_\alpha z(\alpha) 
=\partial_\alpha \arctan\left(\frac{z_2(\alpha)-z_2(\beta)}{z_1(\alpha)-z_1(\beta)}\right) \,,
\end{equation} 
which means that due to \eqref{ss779}, the identity \eqref{ss778} can be integrated, and 
\begin{equation}\label{eq9}
\jump{\phi} =\text{P.V.}\int_\bbR \varpi(\beta)\arctan\left(\frac{z_2(\alpha)-z_2(\beta)}{z_1(\alpha)-z_1(\beta)}\right)d\beta\,.
\end{equation} 
%
%Repeating the process for the lower fluid and computing the corresponding antiderivatives, we find that the jump in the normal direction of the velocity potential $\jump{\phi}$ across the interface $\Gamma(t)$ satisfies
%\begin{equation}\label{eq9}
%\lim_{\epsilon\rightarrow0}\phi^\pm(z(\alpha)\pm\epsilon\partial_\alpha^\perp z(\alpha))=\text{P.V.}\int_\bbR \varpi(\beta)\arctan\left(\frac{z_2(\alpha)-z_2(\beta)}{z_1(\alpha)-z_1(\beta)}\right)d\beta\pm\frac{1}{2}\jump{\phi}(\alpha) \,.
%\end{equation}

We define the Atwood number 
$$
A=\frac{\rho^+-\rho^-}{\rho^++\rho^-} \,.
$$
Using \eqref{eq9} together with Bernoulli's equation \eqref{bernoulli}, we can derive the equation for $\partial_t \jump{\phi}$. 
After computing  $\partial_ \alpha \partial_t \jump{\phi}$, and using \eqref{jump_cond}, \eqref{ss1} and \eqref{z-dynamics}, 
a lengthy computation shows that 
\begin{align}
\varpi_t&=-\partial_\alpha\bigg{[}
 \frac{A}{4\pi^2} \left|\int \varpi( \beta) \frac{(z( \alpha ,t) - z(\beta,t))^\perp}{| z( \alpha ,t) - z(\beta,t)|^2} d\beta\right|^2
 -\frac{A}{4} \frac{\varpi( \alpha ,t)^2}{|{\partial_\alpha} z( \alpha ,t)|^2}
\nonumber\\
&\quad \qquad
+  \frac{A}{\pi}  \int \varpi( \beta) \frac{(z( \alpha ,t) - z(\beta,t))^\perp}{| z( \alpha ,t) - z(\beta,t)|^2}  \cdot \partial_ \alpha z(\alpha ,t) c( \alpha ,t) d\beta 
\nonumber \\
&\qquad \qquad
-c( \alpha ,t) \varpi( \alpha ,t)  - \frac{2\jump{p}}{\rho^++\rho^-}  -2 A g z_2 \bigg{]} \nonumber \\
& \qquad \qquad
+\frac{A}{\pi} \partial_t\bigg{[}
 \int \varpi( \beta) \frac{(z( \alpha ,t) - z(\beta,t))^\perp}{| z( \alpha ,t) - z(\beta,t)|^2}  \cdot \partial_ \alpha z(\alpha ,t)  d\beta  
\bigg{]} \,.
\label{eq11b}
\end{align}

\section{The case that the interface $\Gamma(t)$ is a graph}

We now assume that the interface $\Gamma(t)$ is the graph of the signed height function $h( \alpha ,t)$, so that the parameterization $z( \alpha , t)$ is given by
\begin{equation}\label{h}
\left( z_1 ( \alpha ,t), z_2 (\alpha ,t)\right) = \left(\alpha , h( \alpha ,t) \right) \,.
\end{equation} 

If at time $t=0$, $\Gamma(0)$ is given as the graph $( \alpha , h(\alpha ,0))$, then we can ensure that $\Gamma(t)$ stays a graph for future time $t>0$ by making
an explicit choice of the function $c( \alpha , t)$ in \eqref{eqb}.   To this end, we define
\begin{equation}\label{ss0}
c( \alpha , t) = {\frac{1}{2\pi}}  P.V. \int \varpi( \beta) \frac{ h( \alpha , t) - h( \beta , t) } { (\alpha - \beta )^2 + \left( h(\alpha ,t) - h( \beta , t) \right)^2} d \beta \,.
\end{equation}

\def\BR{\mathcal{K}_{\small \operatorname{BR}}( \alpha,t )  }
\def\p{\partial}
With $\Gamma(t)$ given by $ ( \alpha , h(\alpha ,t)$,  the definition of the Birkhoff-Rott function $\mathcal{K}_{\small \operatorname{BR}}( \alpha,t )$ in (\ref{eq1}) simplifies to
\begin{equation}\label{ss2}
\BR (\alpha ,t) = \frac{1}{2\pi} P.V. \int \frac{ \varpi( \beta ,t)}{ \alpha - \beta } \left( \frac{- \frac{h( \alpha ,t) - h( \beta , t)}{ \alpha - \beta }}{1+ \left(\frac{h( \alpha ,t) - h( \beta ,t)}{\alpha - \beta } \right)^2}  \,,
\frac{1}{1+ \left(\frac{h( \alpha ,t) - h( \beta ,t)}{\alpha - \beta } \right)^2}  \right) \,,
\end{equation} 
the evolution equation for $\Gamma(t)$ can be written in terms of the height function $h( \alpha , t)$ as
\begin{equation}\label{ss3}
h_t( \alpha ,t) =  \frac{1}{2\pi} P.V. \int \frac{ \varpi( \beta ,t)}{ \alpha - \beta } \frac{1}{1+ \left(\frac{h( \alpha ,t) - h( \beta ,t)}{\alpha - \beta } \right)^2} d \beta + 
c(\alpha ,t) \p_\alpha h(\alpha ,t) \,,
\end{equation} 
and  the evolution equation for vorticity $\varpi$ on $\Gamma(t)$ takes the form
\begin{align} 
\varpi_t( \alpha ,t) & =  \frac{A}{4} \p_ \alpha \left(\frac{ \varpi^2( \alpha ,t)}{1+ (\p_\alpha h(\alpha ,t))^2} \right) + \p_ \alpha(c( \alpha ,t) \varpi( \alpha ,t))  
+ \frac{2}{ \rho^++\rho^-} \p_ \alpha \jump{p} \nonumber \\
& \qquad 
+2 A \p_ \alpha  \BR \cdot (1, \p_ \alpha h) c(\alpha ,t) + 2A g \p_ \alpha h + 2 A \p_t \BR \cdot (1, \p_ \alpha h) \,.
\label{ss4}
\end{align} 

\section{A  simple model equation for RT instability}  In this section, we shall derive a model of RT instability in which the interface 
$\Gamma(t)$ is a graph of the height function $h( \alpha ,t)$.
In order to proceed with our derivation,  we shall make further approximations to the coupled system of equations (\ref{ss3}) and (\ref{ss4}).

\subsection{The equations in the linear regime}
In the linear regime, equations (\ref{ss3}) and (\ref{ss4}) reduce to the following coupled system:
\begin{subequations}\label{linear}
\begin{align} 
h_t( \alpha ,t) & =  {\frac{1}{2}} H \varpi \,, \\
\varpi_t( \alpha ,t) & = +2A g \p_ \alpha h  +  \frac{2 \sigma }{ \rho^++\rho^-} \p_ \alpha ^3 h \,,
\end{align} 
\end{subequations}
where $H$ denotes the Hilbert transform, defined as
\begin{equation}\label{Hilbert}
H \varpi( \alpha) = \frac{1}{\pi} P.V. \int \frac{\varpi (\beta)}{ \alpha - \beta } d \beta  \,.
\end{equation}

\def\ao{\langle \varpi \rangle }
\def\aoo{\langle \varpi_0 \rangle }
\subsection{The nonlinear regime and the model equation}  Having found the linear dynamics, we turn our attention to the the nonlinear regime.
Our objective is to derive a  model equation which contains  the quadratic   nonlinearities.  To do so, we shall introduce some further notation.

We define the operator $ \Lambda $ by 
$$
\Lambda \varpi = H \p_ \alpha \varpi \,,
$$
and the space-integrated vorticity as 
$$
\ao(t) = \int \varpi( \beta , t) d \beta \,.
$$
We shall use the notation $ \langle f \rangle (t)$ to denote $ \int f( \beta , t) d \beta $ for any integrable function $f( \beta  ,t)$.

We make use of the following power series expansions for $ | \zeta | < 1$:
$$
\frac{1}{1 + \zeta ^2} = 1 - \zeta ^2 +\zeta ^4 - \zeta ^6 + \cdot\cdot\cdot  \,,
$$
and
$$
\frac{\zeta }{1 + \zeta ^2} = \zeta  - \zeta ^3 +\zeta ^5 - \zeta ^7 + \cdot\cdot\cdot  \,,
$$

Using these identities together with the approximation
$$
\frac{h(\alpha)-h(\beta)}{\alpha-\beta}\approx \partial_\alpha h(\alpha)+\frac{1}{2}\partial_\alpha^2 h(\alpha)(\beta-\alpha), 
$$
the nonlocal terms in equations (\ref{ss3}) and (\ref{ss4}) can be approximated as
$$
c(\alpha,t)\p_ \alpha h(\alpha,t)\approx \frac{\p_ \alpha h(\alpha,t)}{2\pi}P.V. \int \frac{\varpi( \beta)}{\alpha-\beta} \frac{ h( \alpha , t) - h( \beta , t) } { \alpha - \beta} d \beta\,,
$$
\begin{align*}
\BR &\approx \frac{1}{2\pi} P.V. \int \frac{ \varpi( \beta ,t)}{ \alpha - \beta } \left( - \frac{h( \alpha ,t) - h( \beta , t)}{ \alpha - \beta }  \,,
1-\left(\frac{h( \alpha ,t) - h( \beta ,t)}{\alpha - \beta } \right)^2  \right) \,,
\end{align*}
\begin{align*}
\p_ \alpha  \BR \cdot (1, \p_ \alpha h) c(\alpha ,t)&\approx 
-\frac{1}{2}\p_\alpha\left( \frac{1}{2\pi}P.V. \int \frac{\varpi( \beta)}{\alpha-\beta} \frac{ h( \alpha , t) - h( \beta , t) } { \alpha - \beta} d \beta\right)^2\\
&\quad+\frac{1}{2}\Lambda \varpi \partial_\alpha h \frac{1}{2\pi}P.V. \int \frac{\varpi( \beta)}{\alpha-\beta} \frac{ h( \alpha , t) - h( \beta , t) } { \alpha - \beta} d \beta\\
&\approx 
-\frac{1}{2}\p_\alpha\left( \frac{\partial_\alpha h}{2}H\varpi -\frac{\partial_\alpha^2 h}{4\pi}\ao\right)^2+\frac{1}{2}\Lambda \varpi \partial_\alpha h\left( \frac{\partial_\alpha h}{2}H\varpi -\frac{\partial_\alpha^2 h}{4\pi}\ao\right),
\end{align*}
and
\begin{align*}
\p_t  \BR \cdot (1, \p_ \alpha h)&\approx 
\frac{1}{2\pi}P.V. \int \frac{\varpi_t( \beta)}{\alpha-\beta} \left(-\frac{ h( \alpha , t) - h( \beta , t) } { \alpha - \beta} +\p_\alpha h(\alpha)\right)d \beta\\
&\quad-\frac{1}{2}H\varpi\partial_t\partial_\alpha h+\frac{1}{4\pi}\ao\partial_t\partial_\alpha^2 h\\
&\approx 
\frac{1}{4\pi}\partial_\alpha^2 h\frac{d}{dt}\ao-\frac{1}{2}H\varpi\partial_t\partial_\alpha h+\frac{1}{4\pi}\ao\partial_t\partial_\alpha^2 h.
\end{align*}

If we neglect terms of order $O(h^2)$ in equations \eqref{ss3} and \eqref{ss4}, we find that
\begin{subequations}\label{nonlinear1}
\begin{align} 
h_t( \alpha ,t) & =  {\frac{1}{2}} H \varpi \,, \\
\varpi_t( \alpha ,t) & = 2A g \p_ \alpha h  +  \frac{2 \sigma }{ \rho^++\rho^-} \p_ \alpha ^3 h  \nonumber \\
& \qquad 
 +{\frac{A}{2}} \varpi \p_ \alpha \varpi
+\p_ \alpha\left( \frac{\varpi }{2}   \p_ \alpha h \, H \varpi \right)  - \p_ \alpha\left( {\frac{ \varpi}{4\pi}} \p_ \alpha ^2 h\,\right)  \ao \nonumber \\
& \qquad
+ {\frac{A}{4\pi}}  \p_ \alpha \Lambda \varpi \ao
-  \frac{A}{2} H \varpi \Lambda \varpi +\frac{A}{2\pi}\partial_\alpha^2 h(\alpha)\frac{d}{dt}\ao\,.
\end{align} 
\end{subequations}

Integrating in space, we obtain that the vorticity average $\ao$ verifies
$$
\frac{d}{dt}\ao=0;
$$
thus, 
$$
\ao(t)=\ao(0)=\aoo.
$$
Finally, we can use the Tricomi relation for the Hilbert transform,
\begin{equation}\label{tricomi}
2H(fHf)=(Hf)^2-f^2 \,,
\end{equation} 
to obtain
$$
\frac{1}{2}\partial_\alpha((H\varpi)^2-\varpi^2)=H \varpi \Lambda \varpi-\varpi \p_ \alpha \varpi=\partial_\alpha H(\varpi H\varpi).
$$
Then, the system (\ref{nonlinear1}a,b) is equivalent to
\begin{subequations}\label{nonlinear3}
\begin{align} 
h_t( \alpha ,t) & =  {\frac{1}{2}} H \varpi \,, \\
\varpi_t( \alpha ,t) & = 2A g \p_ \alpha h  +  \frac{2 \sigma }{ \rho^++\rho^-} \p_ \alpha ^3 h 
+ \p_ \alpha\left(\frac{\varpi }{2}   \p_ \alpha h \, H \varpi\right)\nonumber\\
&\qquad - \p_ \alpha\left( {\frac{ \varpi}{4\pi}} \p_ \alpha ^2 h\,\right)  \aoo+ {\frac{A}{4\pi}}  \p_ \alpha \Lambda \varpi \aoo
-{\frac{A}{2}} \Lambda(\varpi H \varpi) \,.
\end{align} 
\end{subequations}

The coupled first-order system \eqref{nonlinear3} can be written as one second-order equation for the evolution of the height function:
\begin{align} 
h_{tt}( \alpha ,t)  &= A g \Lambda h  -  \frac{ \sigma }{ \rho^++\rho^-} \Lambda^3 h
- \Lambda\left(H h_t   \p_\alpha h \, h_t\right) \nonumber\\
&\qquad - A \p_\alpha(H h_t  h_t)+ \Lambda\left( {\frac{H h_t}{4\pi}} \p_ \alpha ^2 h\,\right)  \aoo+ {\frac{A}{4\pi}}  \p_ \alpha \Lambda h_t \aoo \,.\label{wavenonlinear1}
\end{align}
The model equation \eqref{wavenonlinear1} contains both quadratic and cubic nonlinearities, but we can simply further.

Keeping only the quadratic nonlinearities, we find that the
graph of the interface $(x, h(x,t))$ evolves according to
\begin{align} 
h_{tt}( \alpha ,t)  &= A g \Lambda h  -  \frac{ \sigma }{ \rho^++\rho^-} \Lambda^3 h
 \nonumber\\
&\qquad - A \p_\alpha(Hh_t h_t)+ \Lambda\left( {\frac{H h_t}{4\pi}} \p_ \alpha ^2 h\,\right)  \aoo+ {\frac{A}{4\pi}}  \p_ \alpha \Lambda h_t \aoo \,.\label{wavenonlinear0}
\end{align} 

By assuming that our initial vorticity has zero average, we arrive at the nonlinear equation
\begin{equation} 
h_{tt}( \alpha ,t)  = A g \Lambda h  -  \frac{ \sigma }{ \rho^++\rho^-} \Lambda^3 h
 - A \p_\alpha(H h_t h_t) \,.\label{wavenonlinear}
\end{equation} 
Equation   \eqref{wavenonlinear} is supplemented with initial conditions.  In particular, we specify the initial interface position
and velocity, respectively, by
$$
h( \cdot , 0) = h_0 \ \text{ and } h_t(\cdot , t) = h_1 \,.
$$
As we explain in the next section, this model equation has a natural stability condition, requiring the product of the Atwood number $A$ and
and the initial velocity field $h_1$ to be positive. 

The second-order-in-time nonlinear wave equation   \eqref{wavenonlinear} is a new model for the motion of an interface under the influence of
RT instability.   We shall refer to either the system  \eqref{nonlinear3} or the wave equation  \eqref{wavenonlinear} as the $h$-model.

%In the same way,
%\begin{align} 
%\p_t^2 \varpi( \alpha ,t)  &= A g \Lambda \varpi  -  \frac{ \sigma }{ \rho^++\rho^-} \Lambda^3 \varpi
%+ \p_t\p_ \alpha\left(\frac{\varpi }{2}   \p_\alpha h \, H \varpi\right)\nonumber\\ 
%&\qquad- \p_t\p_ \alpha\left( {\frac{ \varpi}{4\pi}} \p_ \alpha ^2 h\,\right)  \aoo+ {\frac{A}{4\pi}}  \p_ \alpha \Lambda \p_t\varpi \aoo
%- {\frac{A}{2}} \p_t\Lambda(\varpi H \varpi) \,.\label{wavenonlinear2}
%\end{align} 

\section{Well-posedness of the $h$-model equation \eqref{wavenonlinear}}\label{sec:wp}
We now prove that our  $h$-model equation \eqref{wavenonlinear} is well-posed in Sobolev spaces, when a certain stability condition is 
satisfied by the data.

The space $L^2( \mathbb{T}  )$ consists of the Lebesgue measurable functions on the circle $ \mathbb{T}  $ which are square-integrable with norm
$\| h\|_0^2 = \int_ \mathbb{T}  |h|^2 dx$.  For $s \ge 0  $, 
we define the homogeneous Sobolev space 
$$
\dot{H}^s( \mathbb{T} ) = \left\{ h \in L^2( \mathbb{T}  ) \ : \ \sum_{k \in \mathbb{Z}  }   |k|^{2s} |\hat h(k)|^2  < \infty \right\} \,,
$$
with norm defined as 
\begin{equation}\label{hsnorm}
\|h\|_s^2 = \int_ \mathbb{T}  |\Lambda ^s h|^2  dx \,.
\end{equation} 
For $s \ge 0$, functions $ \dot{H}^s( \mathbb{T} ) $ are identified with $2\pi$-periodic functions  $[-\pi,\pi]$ with finite norm $\| \cdot \|_s$.

From  equation   \eqref{wavenonlinear}, 
$$
\langle h_t (t)\rangle=\langle h_1\rangle.
$$ 
Thus,
$$
\frac{d}{dt}\langle h(t)\rangle=\int_{\bbT} h_t(t)dx=\langle h_1\rangle,
$$
so that
$$
\langle h (t)\rangle=\langle h_0\rangle+\langle h_1\rangle t.
$$ 
These identities, together with the Poincar\'e inequality,  show that \eqref{hsnorm} is an equivalent $H^s( \mathbb{T})  $-norm.  (Recall
that  we are using the notation $ \langle f \rangle (t)$ to denote $ \int f( \beta , t) d \beta $ for any integrable function $f( \beta  ,t)$.)

\begin{theorem}[Local well-posedness for the $h$-model  \eqref{wavenonlinear}]\label{thm1}
Let $\sigma\geq0$, $\rho^+,\rho^->0$, $g\neq0$ be fixed constants and let $(h_0, h_1)$  denote initial position and velocity pair for 
equation \eqref{wavenonlinear}.  Suppose that
$(h_0, h_1)\in H^{2.5+\text{sgn}(\sigma)}(\mathbb{T})\times H^{2}(\mathbb{T})$ and  let
\begin{equation}\label{lambda2}
\lambda:= \min_{\alpha\in \bbT} A\,  h_1(\alpha)
\end{equation}
If
\begin{equation}\label{stability2}
\lambda>0 \,,
\end{equation}
then there exists a time $0<T^*(h_0, h_1)\leq \infty$ and a unique classical solution of \eqref{wavenonlinear} satisfying
$$
h\in C^0([0,T^*];H^{2.5+\text{sgn}(\sigma)}),\;h_t\in C^0([0,T^*];H^2)\cap L^2(0,T^*;H^{2.5}).
$$
\end{theorem}
\begin{proof}
{\it Step 1. Approximate problem for $h^ \epsilon $, $ \epsilon >0$.}  For $ \epsilon >0$, we introduce a sequence of approximations
to equation  \eqref{wavenonlinear}. 
We let $P_\epsilon $ denote the projection operator in $L^2( \mathbb{T})$, given by
$$P_\epsilon f( \alpha ) =   \sum _{|k| \le 1/ \epsilon } \hat f (k) e^{ik \alpha } \,,
$$
where $\hat f(k)$ denotes the $k$th Fourier mode of $f$.
  Then, we let $h^\epsilon $ be a solution to
\begin{equation}\label{approx}
h^ \epsilon _{tt}( \alpha ,t)  = A g P_\epsilon   \Lambda  P_\epsilon h^ \epsilon   -  \frac{ \sigma }{ \rho^++\rho^-} P_\epsilon \Lambda^3  P_ \epsilon h^ \epsilon 
 - A P_\epsilon  \p_\alpha(P_\epsilon h^ \epsilon _t \,  P_\epsilon H h^ \epsilon _t)  \,,
\end{equation} 
with initial data given by $(P_\epsilon h_0, P_\epsilon h_1)$.   The projection operator $P_ \epsilon $ commutes with $\p_ \alpha $ and with $H$ (and hence
with $ \Lambda$).

 We let the parameter $ \epsilon $ range in the interval $ (0, \epsilon _0]$ for a 
constant $  \epsilon _0\ll 1$ to be chosen later.  Then,  ODE theory provides a unique short-time solution 
$h^\epsilon $
to equation \eqref{approx} on a time interval $[0, T_ \epsilon ]$; the solution $h_ \epsilon $ is smooth and can be taken in $ C^ 2([0,T_ \epsilon ]; H^s( \mathbb{T}  ))$ for all $s \ge 0$.

According to the equation \eqref{approx}, $h_{tt}^\epsilon = P_ \epsilon h_{tt}^ \epsilon $; since  $\left(h^ \epsilon (0), h^\epsilon _t(0)\right) = (P_ \epsilon h_0,
P_\epsilon h_1)$, the fundamental theorem of calculus shows that
$$
h^\epsilon = P_ \epsilon h^ \epsilon  \  \text{ and } \  h_t^\epsilon = P_ \epsilon h_t^ \epsilon  \,.
$$
As such, \eqref{approx} can be written as
\begin{equation} 
h^ \epsilon _{tt}( \alpha ,t)  = A g   \Lambda P_ \epsilon ^2  h^ \epsilon   -  \frac{ \sigma }{ \rho^++\rho^-}  \Lambda^3  P_ \epsilon ^2 h^ \epsilon 
 - A P_\epsilon  \p_\alpha( h^ \epsilon _t \,  H h^ \epsilon _t)  \,. \tag{\ref{approx}'}
\end{equation} 

\vspace{.1 in}
\noindent
{\it Step 2. The higher-order energy norm.}  We define the higher-order energy norm $E(t)$ by
\begin{equation}\label{eq:energy}
E(t)=\sup_{0\leq s\leq t}\left\{\frac{\sigma}{\rho^++\rho^-}\|h^ \epsilon (s)\|_{3.5}^2+\|h^ \epsilon _t(s)\|_{2}^2 
+ \frac{\lambda}{2} \int_0^t \|h^\epsilon _t(s)\|_{2.5}^2ds\right\}.
\end{equation}
Note that for each $ \epsilon \in (0, \epsilon _0]$, $t \mapsto E(t)$ is continuous on $[0,T^ \epsilon]$.

\vspace{.1 in}
\noindent
{\it Step 3. Lower bound for $ h^ \epsilon _t$.}  We shall assume that by choosing $T_ \epsilon $  and $\epsilon _0$ sufficiently small, 
\begin{equation}\label{lowerbound}
 A\,h^ \epsilon _t ( \alpha , t) \ge \frac{\lambda}{2} \ \forall \  \alpha \in \mathbb{T}  , t \in [0, T_ \epsilon ], \epsilon \in [0, \epsilon _0] \,.
\end{equation} 
Below, we shall verify that this assumption holds on a time interval $[0,T]$, where $T$ is independent of $ \epsilon$.

\vspace{.1 in}
\noindent
{\it Step 4. $ \epsilon $-independent energy estimates.}   We now establish a time of existence  and bounds  for  $h^ \epsilon $ which are independent
of $ \epsilon$.

We multiply equation  \eqref{wavenonlinear} by $\Lambda^4 h^ \epsilon _t$,  integrate over $ \mathbb{T}  $, and find that
\begin{equation}\label{eq:est}
\frac{1}{2}\left(\frac{d}{dt}\|h^ \epsilon _t\|_{2}^2+\frac{\sigma}{\rho^++\rho^-}\frac{d}{dt}\|h^\epsilon \|_{3.5}^2\right)\leq |g|\|h^\epsilon \|_{2.5} 
\|h^\epsilon _t\|_{2.5}+ I,
\end{equation}
where 
\begin{equation}\label{termI}
I=-A\int_{\bbT}\p_\alpha(h^\epsilon _tHh^ \epsilon _t) \, \p_\alpha^4h^\epsilon _t \, d\alpha.
\end{equation}
We first write the integral $I$ as
\begin{align*}
I&=-A\int_{\bbT}\p_\alpha^3(h^\epsilon_t \, H h^\epsilon_t) \p_\alpha^2 h^\epsilon_td\alpha\nonumber\\
& = \underbrace{-A\int_{\bbT} \Lambda \p_ \alpha ^2 h^\epsilon_t \,  h^\epsilon_t \,  \p_\alpha^2 h^\epsilon_td\alpha}_{I_1}
-
 \underbrace{A\int_{\bbT}  H h^\epsilon_t \,  \p_ \alpha ^3 h^\epsilon_t \,  \p_\alpha^2 h^\epsilon_td\alpha}_{I_2}
 -
 \underbrace{3A\int_{\bbT}  \left[ \Lambda   \p_ \alpha  h^\epsilon_t  \p_ \alpha h^ \epsilon _t
 + \Lambda    h^\epsilon_t  \p_\alpha ^2 \right]\,  \p_\alpha^2 h^\epsilon_td\alpha}_{I_3} \,.
\end{align*} 
The integral $I_2$ has an exact derivative, so upon integration-by-parts, 
$$
I_2 = {\frac{A}{2}} \int_{\bbT}  \Lambda  h^\epsilon_t \,   \p_ \alpha ^2 h^\epsilon_t \,  \p_\alpha^2 h^\epsilon_t \, d\alpha \le 
\| h^\epsilon_t\|_2^2\| h_t^\epsilon \|_{L^ \infty } \le C\| h^\epsilon_t\|_2^2\| h_t^ \epsilon \|_{1.75} \,,
$$
the last inequality following from the Sobolev embedding theorem.   The integral $I_3$ clearly has the same bound.

With the assumption \eqref{lowerbound}, 
the integral $I_1$ provides extra regularity for $h_t^ \epsilon $ as follows:
\begin{align*} 
I_1
&\leq-\int_{\bbT}\left(\frac{1}{4\pi}\int_{\bbT}\frac{\p_\alpha^2  h^\epsilon_t(\alpha)-\p_\alpha^2  h^\epsilon_t(\beta)}{\sin^2\left(\frac{\alpha-\beta}{2}\right)}d\beta\right) A  h^\epsilon_t(\alpha)  \p_\alpha^2  h^\epsilon_t(\alpha)d\alpha   \nonumber\\
&\leq-\frac{A}{8\pi}\int_{\bbT}\int_{\bbT}\frac{\p_\alpha^2  h^\epsilon_t(\alpha)-\p_\alpha^2  h^\epsilon_t(\beta)}{\sin^2\left(\frac{\alpha-\beta}{2}\right)} \left(h^\epsilon_t(\alpha)  \p_\alpha^2  h^\epsilon_t(\alpha)-  h^\epsilon_t(\beta)  \p_\alpha^2  h^\epsilon_t(\beta)\right)d\alpha d\beta \nonumber\\
&\leq-\frac{A}{8\pi}\int_{\bbT}\int_{\bbT}\frac{\left(\p_\alpha^2  h^\epsilon_t(\alpha)-\p_\alpha^2   h^\epsilon_t(\beta)\right)\left( h^\epsilon_t(\alpha)- h^\epsilon_t(\beta)\right)}{\sin^2\left(\frac{\alpha-\beta}{2}\right)} \p_\alpha^2  h^\epsilon_t(\alpha)d\alpha d\beta\nonumber\\
&\quad \qquad \qquad -\frac{A}{8\pi}\int_{\bbT}\int_{\bbT}\frac{\left(\p_\alpha^2   h^\epsilon_t(\alpha)-\p_\alpha^2  h^\epsilon_t(\beta)\right)^2}{\sin^2\left(\frac{\alpha-\beta}{2}\right)} 
 h^\epsilon_t(\beta) d\alpha d\beta  \\
 & = -\frac{A}{8\pi}\int_{\bbT}\int_{\bbT}\frac{\left(\p_\alpha^2   h^\epsilon_t(\alpha)
 -\p_\alpha^2  h^\epsilon_t(\beta)\right)^2}{\sin^2\left(\frac{\alpha-\beta}{2}\right)}  
 h^\epsilon_t(\beta) d\alpha d\beta   \,.
\end{align*}
The last equality follows from the fact that
$$
\int_{\bbT}\int_{\bbT}\frac{\p_\alpha^2  h^ \epsilon _t(\beta)\p_\alpha^2 h^\epsilon _t(\alpha)\left( h^\epsilon _t(\alpha)- h^\epsilon _t(\beta)\right)}{\sin^2\left(\frac{\alpha-\beta}{2}\right)} d\alpha d\beta=0\,.
$$
With our assumed lower bound for $h_t^ \epsilon $ in (\ref{lowerbound}), we find that 
$$
-I_1\geq \frac{\lambda}{2} \frac{1}{8\pi}\int_{\bbT}\int_{\bbT}\frac{\left(\p_\alpha^2   h^\epsilon_t(\alpha)
 -\p_\alpha^2  h^\epsilon_t(\beta)\right)^2}{\sin^2\left(\frac{\alpha-\beta}{2}\right)}  d\alpha d\beta \,.
 $$
Using the computation
\begin{align*} 
\int_\bbT f\Lambda fd\alpha&=\int_{\bbT}\left(\frac{1}{4\pi}\int_{\bbT}\frac{f(\alpha)-f(\beta)}{\sin^2\left(\frac{\alpha-\beta}{2}\right)}d\beta\right) f(\alpha) d\alpha \\ 
&=-\int_{\bbT}\frac{1}{4\pi}\int_{\bbT}\frac{f(\alpha)-f(\beta)}{\sin^2\left(\frac{\alpha-\beta}{2}\right)} f(\beta) d\beta d\alpha \\
&=\frac{1}{8\pi}\int_{\bbT}\int_{\bbT}\frac{(f(\alpha)-f(\beta))^2}{\sin^2\left(\frac{\alpha-\beta}{2}\right)}  d\beta d\alpha \,,
\end{align*}
it follows that
$I_1 \leq-\frac{\lambda}{2}\| h^\epsilon_t\|_{2.5}^2$, and hence
$$
I \leq-\frac{\lambda}{2}\| h^\epsilon_t\|_{2.5}^2 + C\| h^\epsilon_t\|_2^2\| h^ \epsilon_t \|_{1.75}\,. 
$$
Thus,
\begin{equation}\label{eq:est2}
\frac{1}{2}\left(\frac{d}{dt}\|h^ \epsilon _t\|_{2}^2+\frac{\sigma}{\rho^++\rho^-}\frac{d}{dt}\|h^\epsilon \|_{3.5}^2\right)\leq |g|\|h^\epsilon \|_{2.5}\|h^ \epsilon _t\|_{2.5}-\frac{\lambda}{2}\|h^ \epsilon _t\|_{2.5}^2+C\|h^ \epsilon_ t\|_2^2\|h^ \epsilon_t \|_{1.75}.
\end{equation}
The fundamental theorem of calculus shows that
$$
\|h^ \epsilon \|^2_{2.5}\leq \|h_0\|^2_{2.5}+ t\int_0^t\|h^ \epsilon _t(s)\|_{2.5}^2ds \,,
$$
and hence
\begin{align*}
\frac{d}{dt}\|h^\epsilon _t \|_{2}^2+\frac{\sigma}{\rho^++\rho^-}\frac{d}{dt}\|h^ \epsilon \|_{3.5}^2+\frac{\lambda}{2}\|h^ \epsilon _t\|_{2.5}^2
&\leq C\|h^ \epsilon \|_{2.5}^2+C\|h^ \epsilon _t\|_2^4\\
&\leq C\|h_0\|_{2.5}^2+Ct\int_0^t\|h^ \epsilon _t\|_{2.5}^2ds+C\|h^ \epsilon _t\|_2^4 \,.
\end{align*}
It follows that
\begin{equation}\label{gronwall}
\frac{d}{dt} E(t)  \le C \| h^ \epsilon _t\|^2_2 E(t) + C \|h_0\|_{2.5} \le CE(t)^2 + C \|h_0\|_{2.5} \,.
\end{equation} 
By Gronwall's inequality,
\begin{equation}\label{Ebound}
E(t) \le C\left[ \| h_0 \|_{2.5}^2 + \|h_1\|_2^2\right] \,.
\end{equation} 

\vspace{.1 in}
\noindent
{\it Step 5. Verifying the lower-bound on $h_t^ \epsilon $.} 
As the bound (\ref{Ebound}) for $E(t)$ is independent of $ \epsilon $ and $ \sigma $, there exists a time interval $[0,T]$, with $T$ independent of 
$ \epsilon $ and $ \sigma $, such that 
\begin{equation}\label{Ebound2}
E(t) \le C\left[ \| h_0 \|_{2.5}^2 + \|h_1\|_2^2\right]  \ \text{ on } [0,T]  \text{ with } T \text{ independent  of } \epsilon , \sigma \,.
\end{equation} 
 Using the equation (\ref{approx}'), we see that $h^ \epsilon _{tt}$ is bounded in $L^2(0,T; L^2( \mathbb{T}  ))$.  Thus,
$$
\left\|h_t(T)- P_ \epsilon h_1\right\|_{L^2}\leq \int_0^{T} \| h_{tt}(s)\|_{L^2}ds\leq C\, T
$$
Since $\left\|h_t(T)- P_ \epsilon h_1\right\|_{2} \le C$, we can interpolate between the last two inequalities and use the Sobolev embedding theorem to 
conclude that
$$
\left\|h_t(T^2)- P_ \epsilon h_1\right\|_{L^\infty}\leq  C T^{1/8} \,.
$$
By choosing $T$ and $ \epsilon_0$ sufficiently small, and using \eqref{lambda2}, we verify \eqref{lowerbound}.

\vspace{.1 in}
\noindent
{\it Step 6. Existence of solutions.}  
From \eqref{Ebound2}, 
for all $1 < p < \infty $,
\begin{align*} 
&h^ \epsilon \rightharpoonup h \ \text{ in }   W^{1,p}(0,T; H^{2}( \mathbb{T}  )) \,, \\
& h^ \epsilon_t \rightharpoonup h \ \text{ in } L^2(0,T; H^{2.5}( \mathbb{T}  )) \cap  H^1(0,T; L^{2}( \mathbb{T}  )) \cap L^p(0,T; H^2( \mathbb{T}  ))  \,.
\end{align*} 
By using Rellich's theorem, we  can pass to the limit as $ \epsilon \to 0$ in (\ref{approx}').  Since $\| f\|_{L^ \infty } = \lim_{p \to \infty } \|f\|_{L^p}$,
we find  that the limit $h$ is a solution of  \eqref{wavenonlinear} on $[0,T]$ and 
$$
h \in L^ \infty (0,T; H^{2.5}( \mathbb{T}  )) \,, \ h_t \in L^ 2 (0,T; H^{2.5}( \mathbb{T}  )) \cap L^ \infty (0,T; H^2( \mathbb{R}  )) \,.
$$
It is easy to prove that $t \mapsto h( \cdot , t)$ and $t \mapsto h_t( \cdot , t)$ are continuous into $H^{2.5}( \mathbb{T}  )$ and 
$H^{2}( \mathbb{T}  )$, respectively, with respect to the weak topologies on $H^{2.5}( \mathbb{T}  )$ and $H^{2}( \mathbb{T}  )$; furthermore, we can
again find a differential inequality that is almost identical to (\ref{gronwall}) (but this time for the solution $h$ rather than the sequence $h^ \epsilon $). It
follows  that $t \mapsto \|h(t)\|_{2.5}$ and $t \mapsto \|h_t(t)\|_{2.5}$ are continuous, hence 
$$
h \in C^0 (0,T; H^{2.5}( \mathbb{T}  )) \,, \ h_t \in L^ 2 (0,T; H^{2.5}( \mathbb{T}  )) \cap C^0(0,T; H^2( \mathbb{R}  )) \,.
$$
When $ \sigma >0$, by the same argument, we have the better regularity
$$
h \in C^0 (0,T; H^{3.5}( \mathbb{T}  )) \,.
$$
%Integrating in time, we find that
%$$
%E(t)+\frac{\lambda}{8}\int_0^t\|\ h_t\|_{2.5}^2ds\leq E(0)+Ct\|h_0\|_{2.5}^2+Ct^2\int_0^t\| h_t\|_{2.5}^2ds+Ct(E(t))^2
%$$
%By picking $T^0$ small enough, we can absorb the integral in the right hand side and we obtain that
%\begin{equation}\label{energybound2}
%E(t)+\frac{\lambda}{16}\int_0^t\|h_t(s)\|_{2.5}^2ds\leq \mathcal{M}_0+tC(E(t))^2 \,.
%\end{equation}

%\textbf{Time regularity:} When $\sigma>0$, the previous estimates imply 
%$$
%h_t\in L^2(0,T^*;H^{2.5}),
%\;h_{tt}\in L^\infty(0,T^*;H^{0.5})
%$$
%with bounds depending only on the initial data $h_0,h_1$. Thus,
%$$
%h_t\in C([0,T^*],H^{1.5}).
%$$
%Interpolating 
%$$
%H^2\subset H^r\subset L^2,
%$$ 
%we have that 
%$$
%h_t\in C([0,T^*],H^{r})\;\forall\, 0\leq r<2.
%$$
%Similarly,
%$$
%h\in L^\infty(0,T^*;H^{3.5}),
%\;h_t\in L^2(0,T^*;H^{2.5})
%$$
%so
%$$
%h\in C([0,T^*],H^3).
%$$
%We conclude this step by interpolation. When $\sigma=0$, we have 
%$$
%h_t\in L^2(0,T^*;H^{2.5}),
%\;h_{tt}\in L^\infty(0,T^*;H^{1})\cap L^2(0,T^*;H^{1.5}),
%$$
%so 
%$$
%h_t\in C([0,T^*],H^2)
%$$
%and, using the Fundamental Theorem of Calculus,
%$$
%h\in C([0,T^*],H^{2.5}).
%$$

\vspace{.1 in}
\noindent
{\it Step 7. Uniqueness of solutions.}  
 Uniqueness of solutions follows from a standard $L^2$-type energy estimate, and we omit the details.
\end{proof}

\begin{remark} If the stability condition $A\, h_1 >0$ is not satisfied, the evolution equation (\ref{wavenonlinear}) may not be well-posed in Sobolev
spaces.  For analytic initial data, however, the equation does not require a stability condition for a short-time existence theorem.  We shall investigate this
further in future work.
\end{remark}

Having established existence and uniqueness of classical solutions to the RT model \eqref{wavenonlinear}, we next establish a number of interesting
energy laws satisfied by its solutions.

\begin{proposition}[Energy laws for solutions to the $h$-model  \eqref{wavenonlinear}]\label{prop2}
Given constants $\sigma\geq0,\rho^+, \rho^->0$, and $g \neq 0$, suppose that  $h$ is a solution to  \eqref{wavenonlinear} with initial data
 $(h_0,h_1)$. Then $h$ verifies the following energy laws:
 \begin{subequations}\label{energylaw}
\begin{align} 
&\|h_t\|_{0}^2+\frac{\sigma}{\rho^++\rho^-}\|h\|_{1.5}^2-Ag\|h\|_{0.5}^2+\int_0^tD_1(s) ds \nonumber \\
& \qquad \qquad \qquad \qquad \qquad \qquad 
= \|h_1\|_{0}^2+\frac{\sigma}{\rho^++\rho^-}\|h_0\|_{1.5}^2-Ag\|h_0\|_{0.5}^2, \\
&
\|h_t\|_{0.5}^2+\frac{\sigma}{\rho^++\rho^-}\|h\|_{2}^2-Ag\|h\|_{1}^2+\int_0^tD_2(s) ds \nonumber\\
& \qquad \qquad \qquad \qquad \qquad \qquad 
= \|h_1\|_{0.5}^2+\frac{\sigma}{\rho^++\rho^-}\|h_0\|_{2}^2-Ag\|h_0\|_{1}^2, \\
&
\|h_t-\langle h_1\rangle\|_{-0.5}^2+\frac{\sigma}{\rho^++\rho^-}\| h\|_{1}^2-Ag\|h\|_{0}^2+\int_0^tD_3(s) ds+2Ag\int_0^t 2\langle h_1\rangle\langle  h (s)\rangle ds \nonumber\\
& \qquad \qquad \qquad \qquad \qquad \qquad 
= \|h_1-\langle h_1\rangle\|_{-0.5}^2+\frac{\sigma}{\rho^++\rho^-}\|h_0\|_{1}^2-Ag\|h_0\|_{0}^2,
\end{align} 
\end{subequations}
where the terms $D_1$, $D_2$, and $D_3$ are given by
\begin{align*} 
D_1& =2\int_\bbT\text{P.V.}\int_\bbT \frac{A\left(h_t(\alpha)+h_t(\beta)\right)\left(h_t(\alpha)-h_t(\beta)\right)^2}{8\pi\sin^2\left(\frac{\alpha-\beta}{2}\right)}d\alpha d\beta\,, \\
D_2&=A\int_{\mathbb{T}}h_t\left(\left(\Lambda h_t\right)^{2}+\left(\p_\alpha h_t\right)^{2}\right)d\alpha\,, \\
D_3&=2A\int_{\bbT}h_t\left|H h_t\right|^2d\alpha \,.
\end{align*} 
\end{proposition} 

\begin{proof}
The proof reduces to a careful manipulation of the following integrals:
$$
I_1=-A\int_{\bbT}\p_\alpha(h_t Hh_t)h_td\alpha \,,
$$
$$
I_2=-A\int_{\bbT}\p_\alpha(h_t H h_t)\Lambda h_t d\alpha \,,
$$
$$
I_3=-A\int_{\bbT}\p_\alpha(h_t Hh_t)\Lambda^{-1}(h_t-\langle h_1\rangle)d\alpha\,,
$$
where $I_j=-D_j/2$, for $j=1,2,3$.  First,
\begin{align*}
I_1&= A \int_\bbT Hh_t h_t \p_\alpha h_t d\alpha\\
&= - \frac{A}{2} \int_\bbT \Lambda h_t h_t^2  d\alpha\\
&= - \frac{A}{8\pi} \int_\bbT\text{P.V.}\int_\bbT \frac{(h_t(\alpha)-h_t(\beta))^2(h_t(\alpha)+h_t(\beta))}{\sin^2\left(\frac{\alpha-\beta}{2}\right)} d\beta d\alpha.
\end{align*}
Next,
$$
I_2= -A\int_{\bbT}\p_\alpha h_tHh_t \p_\alpha Hh_t hd\alpha-A\int_{\bbT}h_t\left|\Lambda h_t\right|^2d\alpha.
$$
Using the skew-adjointness of the Hilbert transform and Tricomi relation \eqref{tricomi}, 
\begin{equation*} 
 -A\int_{\mathbb{T}}Hh_t\p_\alpha h_t H\p_\alpha h_t d\alpha = A\int_{\mathbb{T}}h_t H\left(\p_\alpha h_t H\p_\alpha h_t\right) d\alpha =\frac{A}{2}\int_{\mathbb{T}}h_t\left(\left(\Lambda h_t\right)^{2}-\left(\p_\alpha h_t\right)^{2}\right)d\alpha,
\end{equation*}
so that
$$
I_2= -\frac{A}{2}\int_{\mathbb{T}}h_t\left(\left(\Lambda h_t\right)^{2}+\left(\p_\alpha h_t\right)^{2}\right)d\alpha.
$$
Finally,
\begin{align*}
I_3&= -A\int_{\bbT}\Lambda^{-1}\p_\alpha(h_t Hh_t)(h_t-\langle h_1\rangle)d\alpha\\
&= A\int_{\bbT}H(h_t Hh_t)(h_t-\langle h_1\rangle)d\alpha\\
&= -A\int_{\bbT}h_t\left|Hh_t\right|^2d\alpha\\
\end{align*}
where we have used the fact that 
$\Lambda^{-1}\p_\alpha=-H$, so that $ H\p_\alpha=\Lambda$.
\end{proof}

\begin{corollary} If $\rho_+ < \rho^-$ so that the Atwood number $A<0$, and if gravity acts downward so that $g>0$, then whenever the initial velocity $h_1$
satisfies the stability condition
$h_1 <0$, then the energy law (\ref{energylaw}b) shows that 
$$
\|h_t(t)\|_{0.5}^2+\frac{\sigma}{\rho^++\rho^-}\|h(t)\|_{2}^2 + |Ag| \| h(t)\|_{1}^2 \ \text{ decays in time for all } \ t\in[0,T] \,.
$$
\end{corollary}

\begin{remark} 
The energy law (\ref{energylaw}a) provides decay for lower-order norms when $A\, h_1 >0$, while the energy law (\ref{energylaw}c) may actually
cause a growth-in-time $t$ which behaves like $t^2$.  There may indeed  be other higher-order energy laws to the model equation (\ref{wavenonlinear}) that have 
yet  to be established.
\end{remark} 

We now  define the average value of a function $h(\alpha)$ on $[-\pi, \pi]$:
$$
\bar{h}:=\frac{\langle h\rangle}{2\pi} =\frac{1}{2\pi}\int_\bbT h(\alpha)d\alpha \,.
$$

\begin{theorem}[Global well-posedness and asymptotic behavior for the $h$-model   \eqref{wavenonlinear}]\label{thm2}
Let $\sigma\geq0$, $\rho^+,\rho^->0$, $g\neq0$ be fixed constants and let $(h_0, h_1)$ denote the initial position and velocity, respectively, for 
the $h$-model \eqref{wavenonlinear}.  Setting
$$
h_2:= h_{tt}( \cdot , 0) = \Lambda h_0+ \sigma \Lambda^3 h_0- \p_\alpha(H h_1 h_1) \,,
$$
and 
with $\lambda$ defined by \eqref{lambda2}, suppose that
$(h_0, h_1)\in H^{2.5+\text{sgn}(\sigma)}(\mathbb{T})\times H^{2}(\mathbb{T})$  is given such that
\begin{equation}\label{stability22}
\lambda>0 \,,
\end{equation}
and
\begin{equation}\label{smallness3}
\|h_2\|_{0.5}^2+\|h_1\|_1^2+\sigma \|h_1\|_{2}^2<\left(\frac{-\bar{h}_1}{5}\right)^2\,.
\end{equation}
Then there exists a unique classical solution of \eqref{wavenonlinear} satisfying
$$
h\in C^0([0,T];H^{2.5+\text{sgn}(\sigma)}),\;h_t\in C^0([0,T];H^2)\cap L^2(0,T;H^{2.5})\;\forall\,0\leq T<\infty.
$$
Furthermore, as $t \to \infty $,  the solution $h( \cdot , t)$ converges to the homogeneous solution $h^\infty=\bar{h}_0+\bar{h}_1 t$; 
specifically,
$$
\limsup_{t\rightarrow\infty}\|h(t)-h^\infty\|_{1+\text{sgn}(\sigma)}+\|h_t(t)-\bar{h}_1 \|_{1+\text{sgn}(\sigma)}= 0.
$$
\end{theorem}
\begin{proof}
Without loss of generality, we consider $\rho^+=0,$ $\rho^-=1$ and $g=1$ so that $A=-1$.   Our analysis will rely on the fact that
$$\langle h_1 \rangle <0 \,,$$
since we have assumed the initial velocity satisfies $h_1<0$.

It is convenient to introduce a new variable $f( \alpha , t)$ given by
$$
f=h- \bar{h}_0-\bar{h}_1 t \ \text{ so that } \ f_t=h_t-\bar{h}_1.
$$
It follows that $f$ satisfies the evolution equation
\begin{equation} 
f_{tt}( \alpha ,t)  + \Lambda f  +   \sigma\Lambda^3 f- \bar{h}_1 \Lambda f_t=
  \p_\alpha(H f_t f_t)  \,,\label{wavenonlinearf}
\end{equation} 
with initial data
$$
f_0=h_0- \bar{h}_0,\;f_1=h_1-\bar{h}_1.
$$
The local existence for $h$ and $h_t$ (and consequently, for $f$ and $f_t$) follows from Theorem \ref{thm1}. Thus, in order to prove that
solutions exist for all time, it remains only to establish estimates which are uniform in time, and the desired asymptotic behavior will be 
established by showing that both $f$ and $f_t$ converge to zero.

With $\lambda$  defined by \eqref{lambda2}, 
 the stability condition \eqref{stability22} for $h_t$ reduces to
\begin{equation}\label{stabilityf}
\sup_{0\leq t}\|f_t(t)\|_{L^\infty}< -\bar{h}_1.
\end{equation}
Furthermore, we have that
$$
\int _\bbT f(\alpha,t) d\alpha=0 \ \text{ and } \ \int _\bbT f_t(\alpha,t) d\alpha=0,\; \forall\, t\geq0.
$$

We test equation \eqref{wavenonlinearf} against $\Lambda f_t$ and obtain that
\begin{align*}
\frac{1}{2}\frac{d}{dt}\left(\|f_t\|_{0.5}^2+\|f\|_1^2+\sigma \|f\|_{2}^2\right)- \bar{h}_1 \|f_t\|_{1}^2&=\int_{\bbT} \p_\alpha (H f_t f_t) \Lambda f_t d\alpha\\ 
&\leq \|f_t\|_1^2(\|f_t\|_{L^\infty}+ \|Hf_t\|_{L^\infty}).
\end{align*}

We shall make use of the  refined Carlson inequality:
\begin{equation}\label{Carlson}
\|w\|_{L^\infty}^2\leq \|w\|_{0}\|w\|_{1}-\frac{1}{\pi}\|w\|_0^2 \ \ \forall w \in  \dot{H}^1( \mathbb{T} ) \ \text{ with } \ \langle w \rangle =0 \,.
\end{equation}
Using \eqref{Carlson} together with the Poincar\'e inequality shows that

$$
\|w\|_{L^\infty}<\|w\|_{1}.
$$

As a consequence, we obtain that
\begin{equation}\label{energyest}
\frac{1}{2}\frac{d}{dt}\left(\|f_t\|_{0.5}^2+\|f\|_1^2+\sigma \|f\|_{2}^2\right)- \bar{h}_1 \|f_t\|_{1}^2< 2\|f_t\|_1^3.
\end{equation}
Taking a time derivative of equation  \eqref{wavenonlinearf},  we find that $f_t$ satisfies
\begin{equation}\label{ftt}
f_{ttt}( \alpha ,t)  + \Lambda f_t  +   \sigma\Lambda^3 f_t-\bar{h}_1 \Lambda f_{tt}=
  \p_\alpha(H f_{tt} f_t+H f_t f_{tt}),
\end{equation} 
with initial data given by
$$
f_1(\alpha)=f_t(\alpha,0),\;f_2=f_{tt}(\alpha,0)=-\Lambda f_0- \sigma \Lambda^3 f_0+ \bar{h}_1\Lambda f_1 + \p_\alpha(H f_1 f_1) \,.
$$
Testing \eqref{ftt} against $\Lambda f_{tt}$, and using \eqref{Carlson}, we obtain that
\begin{align}
\frac{1}{2}\frac{d}{dt}\left(\|f_{tt}\|_{0.5}^2+\|f_t\|_1^2+\sigma \|f_t\|_{2}^2\right)- \bar{h}_1\|f_{tt}\|_{1}^2&\leq \|f_{tt}\|_1^2(\|f_t\|_{L^\infty}+ \|Hf_t\|_{L^\infty})\nonumber\\
&\quad +\|f_{tt}\|_{1}\|f_t\|_{1}(\|f_{tt}\|_{L^\infty}+ \|Hf_{tt}\|_{L^\infty})\nonumber\\
&< 4\|f_{tt}\|_1^2\|f_t\|_{1}\label{energyest2}.
\end{align}

Next, we define the following energy and dissipation functions, respectively, as
$$
E(t)=\|f_{tt}\|_{0.5}^2+\|f_t\|_1^2+\sigma \|f_t\|_{2}^2,\ \ D(t)=- 2\bar{h}_1 \|f_{tt}\|_{1}^2 \,.
$$
Then \eqref{energyest2} can be written as
$$
\frac{d}{dt}E(t)+D(t)\leq \frac{4\sqrt{E(t)}}{-\bar{h}_1 }D(t),
$$
and we find decay of the energy 
$$
E(t)+\left(1-\frac{4\sqrt{E(0)}}{-\bar{h}_1 }\right) \int_0^t D(s)ds\leq E(0),
$$
provided that the initial data satisfies
$$
E(0)<\left(\frac{-\bar{h}_1}{4}\right)^2.
$$

Consequently, due to \eqref{smallness3},  the initial data $f_0$ and $f_1$ satisfy
\begin{equation}\label{smallness}
E(0)=\|f_2\|_{0.5}^2+\|f_1\|_1^2+\sigma \|f_1\|_{2}^2<\left(\frac{-\bar{h}_1}{5}\right)^2,
\end{equation}
and we have that equations \eqref{energyest} and \eqref{energyest2} reduce to
$$
\frac{d}{dt}\left(\|f_t\|_{0.5}^2+\|f\|_1^2+\sigma \|f\|_{2}^2\right)-\bar{h}_1\|f_t\|_1^2<0,
$$
$$
\frac{d}{dt}\left(\|f_{tt}\|_{0.5}^2+\|f_t\|_1^2+\sigma \|f_t\|_{2}^2\right)-\frac{2\bar{h}_1}{5}\|f_{tt}\|_1^2<0.
$$
Thus, we find the estimates
$$
\|f_t(t)\|_{0.5}^2+\|f(t)\|_1^2+\sigma \|f(t)\|_{2}^2-\bar{h}_1\int_0^t\|f_t(s)\|_1^2ds\leq \|f_1\|_{0.5}^2+\|f_0\|_1^2+\sigma \|f_0\|_{2}^2,
$$
$$
\|f_{tt}(t)\|_{0.5}^2+\|f_t(t)\|_1^2+\sigma \|f_t(t)\|_{2}^2-\frac{2\bar{h}_1}{5}\int_0^t\|f_{tt}(s)\|_1^2ds\leq \|f_{2}\|_{0.5}^2+\|f_1\|_1^2+\sigma \|f_1\|_{2}^2.
$$
and the asymptotic behavior

\begin{equation}\label{decayf}
\limsup_{t\rightarrow \infty} \|f(t)\|_{1+\text{sgn}(\sigma)}+\|f_t(t)\|_{1+\text{sgn}(\sigma)}+\|f_{tt}(t)\|_{0.5}=0.
\end{equation}

Finally, using \eqref{Carlson} and \eqref{smallness}, we conclude that
\begin{equation}\label{smallness2}
\|f_t(t)\|_{L^\infty}^2<\|f_t(t)\|_{1}^2<\left(\frac{-\bar{h}_1}{5}\right)^2,
\end{equation}
and the stability condition \eqref{stabilityf} is satisfied even in a stricter sense. Finally, we test equation \eqref{wavenonlinearf} against $\Lambda^4 f_t$. We find that 
$$
\frac{1}{2}\frac{d}{dt}\left(\|f_t\|_{2}^2+\|f\|_{2.5}^2+\sigma \|f\|_{3.5}^2\right)- \bar{h}_1\|f_t\|_{2.5}^2=\int_\bbT \p_\alpha(f_t H f_t)\partial_\alpha^4 f_td \alpha.
$$
Using integration by parts, Sobolev embedding and interpolation, we have that
\begin{align*}
I&=\int_\bbT \p_\alpha(f_t H f_t)\partial_\alpha^4 f_t d\alpha\\
&=-\int_\bbT \p_\alpha^2(f_t H f_t)\partial_\alpha^3 f_t d\alpha\\
&=\frac{1}{2}\int_\bbT \p_\alpha^2 f_t \Lambda f_t\partial_\alpha^2 f_t d\alpha-2\int_\bbT \p_\alpha f_t \Lambda f_t\partial_\alpha^3 f_t d\alpha+\int_\bbT \Lambda^{0.5}(f_t \p_\alpha\Lambda f_t)\partial_\alpha\Lambda^{1.5} f_t d\alpha\\
&\leq C\|f_t\|_{2.25}^2\|f_t\|_1+\int_\bbT [\Lambda^{0.5},f_t]\p_\alpha\Lambda f_t\partial_\alpha\Lambda^{1.5} f_t+\int_\bbT f_t \p_\alpha\Lambda^{1.5} f_t\partial_\alpha\Lambda^{1.5} f_t d\alpha\\
&\leq C\|f_t\|_{2.5}\left(\|f_t\|_{2}\|f_t\|_1+\|\partial_\alpha f_t\|_{L^4}\|\p_\alpha \Lambda^{0.5}f_t\|_{L^4}+\|\Lambda^{0.5}f_t\|_{L^4}\|\p_\alpha^2 f_t\|_{L^4}\right)+\|f_t\|_{L^\infty}\|f_t\|_{2.5}^2\\
&\leq C\|f_t\|_{2.5}\left(\|f_t\|_{2}\|f_t\|_1+\|f_t\|_{0.75}\|f_t\|_{2.25}\right)+\|f_t\|_{L^\infty}\|f_t\|_{2.5}^2,
\end{align*}
where we have used the classical commutator estimate \cite{kenig1991well}
$$
\|[\Lambda^{0.5},u]v\|_{L^2}\leq C(\|\partial_\alpha u\|_{L^4}\|\Lambda^{-0.5}v\|_{L^4}+\|\Lambda^{0.5}u\|_{L^4}\|v\|_{L^4}).
$$
Using Young's inequality and \eqref{smallness2}, we find that
\begin{equation}\label{energyest3}
\frac{d}{dt}\left(\|f_t\|_{2}^2+\|f\|_{2.5}^2+\sigma \|f\|_{3.5}^2\right)+\delta \|f_t\|_{2.5}^2\leq C_1(h_0,h_1)\|f_t\|_2^2,
\end{equation}
for a certain explicit $0<\delta (h_0,h_1)$. Using Gronwall's inequality, we conclude that
\begin{equation}\label{energyest4}
\|f_t(t)\|_{2}^2+\|f(t)\|_{2.5}^2+\sigma \|f(t)\|_{3.5}^2\leq C_2(h_0,h_1)e^{C_1(h_0,h_1)t}.
\end{equation}
Finally, from the regularity of $f$ and $f_t$, we can conclude the regularity of $h$ and $h_t$. 
\end{proof}

\begin{remark}
Let us emphasize that the condition \eqref{smallness3} does \emph{not} require \emph{small} initial data;  rather, we require the data to be
sufficiently close to an arbitrarily large homogeneous state. For example, in the case that surface tension $\sigma=0$, we can consider
$$
h_0=A+Be^{\alpha i} \ \text{ and } \ h_1=-1000+\frac{e^{\alpha i}}{6} \,
$$ 
for constants $A$ and $B$.
A simple computation using the explicit form of $h_0$ and $h_1$ shows that the condition \eqref{smallness3} is satisfied when $B\leq 110$ and any $A$. 
\end{remark} 

\begin{remark} With  $h_2= \Lambda h_0+ \sigma \Lambda^3 h_0- \p_\alpha(H h_1 h_1) $, our asymptotic behavior requires that
$\|h_2\|_{0.5}^2+\|h_1\|_1^2+\sigma \|h_1\|_{2}^2<\left(\frac{-\bar{h}_1}{5}\right)^2$.   Let us set surface tension $ \sigma =0$, in which case
our initial position $h_0$ and initial velocity $h_1$ must be chosen so that
$$
\| \Lambda h_0- \p_\alpha(H h_1 h_1) \|_{0.5}^2+\|h_1\|_1^2 <\left(\frac{-\bar{h}_1}{5}\right)^2 \,.
$$
Thus, a constraint is placed on the $H^{1.5}$-norm of the initial position function $h_0$, even though the nonlinearity acts only on the 
velocity field $h_1$.  As can be seen in the proof of  Theorem \ref{thm2}, the need to obtain a sign on the derivative of the energy function 
leads to this constraint, but there is an interesting pointwise argument which also explains the need to constrain the size of $\Lambda h_0$.

In order to be able to continue our solution for all time, it is necessary to ensure that the initial stability condition $h_1 < 0$ is propagated,
and that we thus maintain $h_t( \cdot , t) < 0$ for all $t \ge 0$ in which case we must also have that
$$
\| \Lambda h( \cdot , t) - \p_\alpha(H h_t( \cdot ,t) \,  h_t(\cdot , t)) \|_{0.5}^2+\|h_t(\cdot,t)\|_1^2  <\left(\frac{-\bar{h}_1}{5}\right)^2 \ \ \forall t\ge 0\,.
$$

To give another (and perhaps more transparent) explanation of how this can be achieved, we introduce  the time-dependent point
$x_t$ which satisfies
$$
h_t(x_t,t)=\max_\alpha h_t(\alpha,t)\,,
$$
and we define the time-dependent function $y(t)$ by
$$
y(t)=h_t(x_t,t)-\bar{h}_1 \ \text{ so that } \ y(t) \ge 0 \,.
$$
The  Hilbert transform, acting on periodic $2\pi$-periodic functions,  is defined by
$$
\Lambda h(\alpha)=H \partial_\alpha h(\alpha)=\frac{1}{2\pi}\text{P.V.}\int_{\bbT}\frac{\partial_\alpha h(\alpha-y)}{\tan(y/2)}dy.
$$
We time-differentiate this formula and use integration-by-parts (following the computation in  \cite{granero2014global}):
\begin{align*}
\Lambda h_t(x_t)&=\frac{1}{2\pi}\text{P.V.}\int_{\bbT}\frac{\partial_y (h_t(\alpha)-h_t(\alpha-y))}{\tan(y/2)}dy\\
&=\frac{1}{4\pi}\text{P.V.}\int_{\bbT}\frac{h_t(\alpha)-h_t(\alpha-y)}{\sin^2(y/2)}dy\\
&=\frac{1}{4\pi}\text{P.V.}\int_{\bbT}\frac{h_t(x_t)-h_t(y)}{\sin^2((x_t-y)/2)}dy\\
&\geq \frac{1}{4\pi}\left(2\pi h_t(x_t)-\langle h_1\rangle\right)\\
&=\left(h_t(x_t)-\bar{h}_1\right)/2.
\end{align*}
Thus, we have that
\begin{equation}\label{gs1}
\frac{d}{dt}h_t(x_t,t)=\Lambda h(x_t,t)+\Lambda h_t(x_t,t) h_t(x_t,t)\leq \Lambda h(x_t,t)+\frac{h_t(x_t,t)-\bar{h}_1}{2}h_t(x_t,t) \,. 
\end{equation} 
From \eqref{gs1}, it then follows that
$$
\frac{dy}{dt}\leq  \Lambda h(x_t,t) +y(t)^2 +\frac{\bar{h}_1}{2} y(t)\,.
$$
Now, the condition that $h_t( \cdot , t) < 0$ is equivalent to showing that
\begin{equation}\label{gs2}
0\leq y(t)< -\bar{h}_1 \ \ \forall t \ge 0 \,.
\end{equation} 
From the ODE \eqref{gs1}, we see that in order for \eqref{gs2} to hold, we must place a size restriction on $\Lambda h(x_t,t)$.
This size restriction, in turn, requires $L^ \infty $-control of $\Lambda h(x_t,t)$;  in particular, we must have that 
$16\Lambda h(x_t,t) \le \bar{h}_1^2$.
 It is, therefore, somewhat
 remarkable that  our proof of Theorem \ref{thm2} requires only control on $\|\Lambda h_0\|_{0.5}$ rather than $\|\Lambda h_0\|_{L^\infty}$.
\end{remark} 
\begin{remark} 
Finally, we note that the asymptotic condition \eqref{decayf} remains true for higher-order Sobolev norms if the condition \eqref{smallness2} is replaced with further constraints on the initial data involving higher-order time-derivatives of $h$ evaluated at $t=0$. In particular, we do not claim that  \eqref{energyest4} is sharp. 
\end{remark}

\section{General interface parameterization with interface turn-over}

Our  $h$-model equation \eqref{wavenonlinear} for the evolution of the height function $h( \alpha ,t)$ uses a   special parameterization in which the interface $\Gamma(t)$
is constrained to be the graph $( \alpha , h(\alpha , t))$.   While this model works well in predicting the mixing layer, in the unstable regime and when 
the RT instability is initiated, the height function $h( \alpha , t)$ can only grow in amplitude.    
Our goal is to generalize the $h$-model equation \eqref{wavenonlinear} by using a general parameterization $z( \alpha ,t)$ that permits
the interface to  turn-over.   As
we will show, the ability for the wave to turn-over, rather than only grow in amplitude, provides an even more accurate prediction of the RT mixing layer,
at the expense of a slightly more complicated system of evolution equations.

%\subsection{RT model with interface turn-over}
We now return to the general evolution equations
 \eqref{z-dynamics} and \eqref{eq11b} and set $c(\alpha,t)=0$; hence, the interface parameterization
 $z(\alpha,t)=(z_1(\alpha,t),z_2(\alpha,t))$ evolves according to
\begin{equation}\label{z-dynamics2}
z_t(  \alpha ,t) = {\frac{1}{2\pi}}  \int \varpi( \beta) \frac{(z( \alpha ,t) - z(\beta,t))^\perp}{| z( \alpha ,t) - z(\beta,t)|^2} d\beta  \,,
\end{equation}  
while the vorticity amplitude satisfies 
\begin{align}
\varpi_t&=-\partial_\alpha\bigg{[}
 A\left|\frac{1}{2\pi} \int \varpi( \beta) \frac{(z( \alpha ,t) - z(\beta,t))^\perp}{| z( \alpha ,t) - z(\beta,t)|^2} d\beta\right|^2
 -\frac{A}{4} \frac{\varpi( \alpha ,t)^2}{|\partial_\alpha  z( \alpha ,t)|^2}
 - \frac{2\jump{p}}{\rho^++\rho^-}  -2 A g z_2 \bigg{]} \nonumber \\
& \qquad \qquad
+\frac{A}{\pi} \partial_t\bigg{[}
 \int \varpi( \beta) \frac{(z( \alpha ,t) - z(\beta,t))^\perp}{| z( \alpha ,t) - z(\beta,t)|^2}  \cdot \partial_ \alpha z(\alpha ,t)  d\beta  
\bigg{]} \,.
\label{eq11b2}
\end{align}

Using the difference quotient approximation for the derivative, 
$$
\frac{z(\alpha)-z(\beta)}{\alpha-\beta}\approx \partial_\alpha z(\alpha),
$$
we obtain the approximation
\begin{equation}\label{approx1}
{\frac{1}{2\pi}}  \int \varpi( \beta) \frac{(z( \alpha ,t) - z(\beta,t))^\perp}{| z( \alpha ,t) - z(\beta,t)|^2} d\beta \approx \frac{1}{2}H\varpi(\alpha) \frac{(\partial_\alpha z(\alpha))^\perp}{|\partial_\alpha z(\alpha)|^2}.
\end{equation} 
Substitution of \eqref{approx1} into
 \eqref{z-dynamics2} and \eqref{eq11b2}, and using the Tricomi relation \eqref{tricomi}, we obtain the general interface RT model evolution equations which
 allow for wave turn-over:
\begin{subequations}\label{z-dynamics3}
\begin{align}
z_t(  \alpha ,t) & = \frac{1}{2}H\varpi(\alpha,t) \frac{(\partial_\alpha z(\alpha,t))^\perp}{|\partial_\alpha z(\alpha,t)|^2}  \,, \label{z-dynamics2b} \\
\varpi_t(  \alpha ,t) &= -\partial_\alpha\bigg{[}\frac{A}{2} \frac{1}{|\partial_\alpha z( \alpha ,t)|^2}H\left(\varpi(\alpha,t)H\varpi( \alpha ,t)\right)
 - \frac{2\jump{p}}{\rho^++\rho^-}  -2 A g z_2 \bigg{]}\,.  \label{eq11b2b}
\end{align}
\end{subequations}
Substituting the 
time-derivative of \eqref{z-dynamics2b} into   \eqref{eq11b2b} yields 
\begin{align*}
z_{tt}&= -\Lambda\bigg{[}\frac{A}{4} \frac{1}{|\partial_\alpha  z|^2}H\left(\varpi H\varpi\right)
 - \frac{\jump{p}}{\rho^++\rho^-}  - A g z_2 \bigg{]} \frac{(\partial_\alpha z)^\perp}{|\partial_\alpha z|^2} \nonumber\\ 
 &\quad +\frac{1}{2}H\varpi\left(\frac{(\partial_\alpha z_t)^\perp}{|\partial_\alpha z|^2}-\frac{(\partial_\alpha z)^\perp 2(\partial_\alpha z\cdot \partial_\alpha z_t)}{|\partial_\alpha z|^4}\right)\,.
\end{align*}
Notice that \eqref{z-dynamics2b} implies that
\begin{equation}\label{aux1}
H\varpi=2z_t\cdot (\partial_\alpha z)^\perp\text{ and }\varpi=-2H(z_t\cdot (\partial_\alpha z)^\perp).
\end{equation}
Thus, we obtain the equivalent second-order nonlinear wave system for $z( \alpha , t)$ given by
\begin{align}
z_{tt}&= \Lambda\bigg{[}\frac{A}{|\partial_\alpha z|^2}H\left(z_t\cdot (\partial_\alpha z)^\perp H(z_t\cdot (\partial_\alpha z)^\perp)\right)
 + \frac{\jump{p}}{\rho^++\rho^-}  + A g z_2 \bigg{]} \frac{(\partial_\alpha z)^\perp}{|\partial_\alpha z|^2} \nonumber\\ 
 &\quad +z_t\cdot (\partial_\alpha z)^\perp\left(\frac{(\partial_\alpha z_t)^\perp}{|\partial_\alpha z|^2}-\frac{(\partial_\alpha z)^\perp 2(\partial_\alpha z\cdot \partial_\alpha z_t)}{|\partial_\alpha z|^4}\right)\,.\label{eqztt}
\end{align}

We shall refer to the system \eqref{z-dynamics3} or the wave equation \eqref{eqztt}  as the $z$-model.  The $z$-model \eqref{eqztt}
is analogous to the slightly simpler $h$-model \eqref{wavenonlinear}.

\begin{remark} Note well that due to our choice of setting  $c(\alpha,t)=0$,  the evolving interface solving equation \eqref{z-dynamics3} (or equivalently \eqref{eqztt}) does \emph{not} 
remain a graph,  even if the initial data $z(\alpha,0)=(\alpha,h_0(\alpha)), z_t(\alpha,0)=(0,h_1(\alpha))$ are given as graphs.
In particular, it is  convenient (especially for the purposes of comparing against the $h$-model (\ref{wavenonlinear})) to prescribe the initial interface
position as a graph, and allow the interface to evolve into a non-graph state.  As we will show, the $z$-model \eqref{z-dynamics3} captures
the turn-over of the RT interface.
\end{remark}

\section{Numerical study}\label{simus}
\subsection{The algorithm}\label{algorithm}
In order to stabilize numerical oscillations without affecting the amplitude or speed of wave propagation, we shall employ an arbitrary-order artificial viscosity operator
for both the $h$-model \eqref{nonlinear3} (or \eqref{wavenonlinear}) and  the $z$-model \eqref{z-dynamics3} (or \eqref{eqztt}). 

\subsubsection{Numerical approximation of the $h$-model}
 We first consider the $h$-model, written as a system in \eqref{nonlinear3}.  We numerically discretize the following approximation:
\begin{subequations}\label{nonlinearreg}
\begin{align} 
h_t^\epsilon & =  {\frac{1}{2}} H \varpi^\epsilon \,, \\
\varpi^\epsilon_t & = 2A g \p_ \alpha h^\epsilon  +  \frac{2 \sigma }{ \rho^++\rho^-} \p_ \alpha ^3 h^\epsilon 
 - \p_ \alpha\left( {\frac{ \varpi^\epsilon}{4\pi}} \p_ \alpha ^2 h^\epsilon\,\right)  \aoo \nonumber \\%+\p_ \alpha\left(\frac{\varpi^\epsilon }{2}   \p_ \alpha h^\epsilon \, H %\varpi^\epsilon\right)\nonumber\\
&\qquad + {\frac{A}{4\pi}}  \p_ \alpha \Lambda \varpi^\epsilon \aoo
-{\frac{A}{2}} \Lambda(\varpi^\epsilon H \varpi^\epsilon) -\epsilon \Lambda^{s} \varpi^\epsilon\,,
\end{align} 
\end{subequations}
where $\epsilon>0$ is the artificial viscosity, and  $s\geq2$ determines the order of the artificial viscosity employed.   

Equivalently,  we have the approximation for the wave equation given by
\begin{align} 
\p_t^2 h^\epsilon +\epsilon \Lambda^{s} h_t^\epsilon&= A g \Lambda h^\epsilon  - 
 \frac{ \sigma }{ \rho^++\rho^-} \Lambda^3 h^\epsilon - A \p_\alpha(Hh_t^\epsilon h_t^\epsilon)\nonumber\\
%- \Lambda\left(Hh_t^\epsilon   \p_\alpha h^\epsilon \, h_t^\epsilon\right) 
&\qquad + \Lambda\left( {\frac{H h_t^\epsilon}{4\pi}} \p_ \alpha ^2 h^\epsilon\,\right)  \aoo+ {\frac{A}{4\pi}}  \p_ \alpha \Lambda h_t^\epsilon \aoo \,.\label{wavenonlinearreg}
\end{align} 
The term $\epsilon \Lambda^{s} h_t^\epsilon$ represents the artificial viscosity operator.  The parameter $s$ determines the order of the operator; for
example, for $s=2$, we recover the classical Laplace operator, while for $s>2$, we can study a variety of hyperviscosity operators. We believe that
this equation will be an ideal candidate for the $C$-method artificial viscosity which is localized in both space and time (see \cite{ReSeSh2013}), 
and shall implement this in future work.

We use a Fourier collocation method to solve \eqref{nonlinearreg}.   We note that for the numerical simulations that we consider herein, we choose
initial data for which $\aoo=0$, in which case the equation that we discretize is equivalent to
\begin{align} 
\p_t^2 h^\epsilon  +\epsilon \Lambda^{s} h_t^\epsilon&= A g \Lambda h^\epsilon  - 
 \frac{ \sigma }{ \rho^++\rho^-} \Lambda^3 h^\epsilon - A \p_\alpha(Hh_t^\epsilon h_t^\epsilon) \,.
\tag{\ref{wavenonlinearreg}'}
\end{align}

For $N = 0, 1,2, ...$,
we define the our Fourier approximation using our projection operator $P_N$ (introduced in Section \ref{sec:wp}), defined as
$$
P_N f( \alpha )=\sum_{-N}^N\hat{f}(k)e^{ik\alpha }  \,.
$$ 
Hence,  the mesh size of our algorithm is given by
$$
dx=\frac{2\pi}{N} \,.
$$

We make use of the following identities in frequency space:
\begin{align*} 
\widehat{\partial_\alpha^n f}(k) & =(i k)^n\hat{f}(k) \,, \\
\widehat{H f}(k) & =-i \text{sgn}(k)\hat{f}(k) \,, \\
\widehat{\Lambda f}(k)&=-i \text{sgn}(k)i k\hat{f}(k)=|k|\hat{f}(k) \,, \\
\widehat{fg}(k)&=\hat{f}*\hat{g}(k) \,.
\end{align*} 
It follows that in  frequency space,  the system \eqref{nonlinearreg} can be written as
\begin{align*} 
\frac{d}{dt} \hat{h}^\epsilon & =  {\frac{-i \text{sgn}(k)}{2}} \hat{\varpi}^\epsilon \,, \\
\frac{d}{dt}  \hat{\varpi}^\epsilon & = 2A g ik \hat{h}^\epsilon  +  \frac{2 \sigma }{ \rho^++\rho^-} (ik) ^3 \hat{h}^\epsilon 
 - \aoo ik\left( \frac{ \varpi^\epsilon}{4\pi} (-k^2 \hat{h}^\epsilon)\check{}\,\right)\hat{}   \\%+\p_ \alpha\left(\frac{\varpi^\epsilon }{2}   \p_ \alpha h^\epsilon \, H %\varpi^\epsilon\right)\\
&\qquad + {\frac{A}{4\pi}}  ik|k| \hat{\varpi}^\epsilon \aoo
-{\frac{A}{2}} |k|\left(\varpi^\epsilon \left(-i \text{sgn}(k) \hat{\varpi}^\epsilon\right)\check{}\right)\hat{} -\epsilon |k|^{3} \hat{\varpi}^\epsilon\,.
\end{align*} 
This is a nonlinear system of ordinary differential equations, to which we shall apply the adaptive Runge-Kutta-Fehlberg fourth-order (nominally fifth-order) scheme to advance in time-increments .
%(\emph{i.e.} if $dt$ denotes the time step, our time integrator is $O(dt^4)$ and error estimator of $O(dt^5)$).

For the following  simulations, we have set the acceleration to a constant value of $g$, which will either act downward or upward, depending on
the type of RT instability that we examine.   The initial position of the interface  is specified, as well as the initial
amplitude of the vorticity, $\varpi(\alpha,0)$.

\subsubsection{Numerical approximation of the $z$-model}
 We change variables and consider that the curve $z(\alpha,t)$ is given by
$$
z(\alpha,t)=(\alpha+\delta z_1(\alpha,t),z_2(\alpha,t)),
$$
where $\delta z_1, z_2$ are periodic functions defined on $[-\pi,\pi]$.   

For our numerical simulations, we shall only consider the case of 
zero surface tension. 
We use the same Fourier-collocation method described previously (with $N$ Fourier modes) to approximate the following system of equations:
\begin{subequations}\label{nonlinearregZ}
\begin{align} 
(\delta z_1)_t^\epsilon & =  -{\frac{1}{2}} H \varpi^\epsilon \frac{\partial_\alpha z_2^\epsilon}{|(1+\p_\alpha \delta z_1^\epsilon,\partial_\alpha z^\epsilon_2)|^2} +\epsilon\p_\alpha^2 \delta z_1^\epsilon\,, \\
(z_2)_t^\epsilon & =  {\frac{1}{2}} H \varpi^\epsilon \frac{1+\partial_\alpha \delta z_1^\epsilon}{|(1+\p_\alpha \delta z_1^\epsilon,\partial_\alpha z_2^\epsilon)|^2} +\epsilon\p_\alpha^2 z_2^\epsilon\,, \\
\varpi_t^\epsilon & =-\partial_\alpha\bigg{[}\frac{A}{2} \frac{1}{|(1+\p_\alpha \delta z_1^\epsilon,\partial_\alpha z^\epsilon_2)|^2}H\left(\varpi^\epsilon H\varpi^\epsilon\right)
  -2 A g z^\epsilon_2 \bigg{]} +\epsilon \p_\alpha^2 \varpi^\epsilon\,.
\end{align} 
\end{subequations}
The system \eqref{nonlinearregZ}  employs an artificial viscosity  term $\epsilon\partial_\alpha^2$)  to stabilize small-scale noise.

\subsection{Simulation 1: $h$-model,  $\frac{\rho^+}{\rho^-}= {\frac{1}{1.5}}$ and Atwood number $A<0$}\label{sim1}  We first study the effect of the density ratio on the interface motion given by the simple RT $h$-model \eqref{nonlinearreg} in the absence of surface tension ($\sigma=0$). For our first simulation, we consider  two fluids with density ratio  $1/1.5$.

Specifically, we consider the physical parameters set to
$$
g=9.8\,m/s,\;\sigma=0,\;\rho^+=1\, kg/m^3,\;\rho^-=1.5\,kg/m^3,\text{ so }A=-0.2\,,$$
with initial data 
\begin{align} 
h(\alpha,0)&=\sin(3\alpha)\label{initialdatanum} \,,\\
\varpi(\alpha,0)&=2H\sin(2\alpha),\label{initialdatanum2} \,.
\end{align}

In order to study convergence of our scheme and the effect of the artificial viscosity operator, we perform a number of different simulations,
varying the 
total number of modes $N$ that are used,  as well as the power on the artificial viscosity operator $s$ and the size of the artificial viscosity
parameter $\epsilon$. 
In particular, we consider the following cases:
\begin{itemize}
\item $N=2^7$, $\epsilon=0.01$, $s=3$ (red solid line in the Figures \ref{fig1} and \ref{fig2})
\item $N=2^7$, $\epsilon=0.008$, $s=3$ (blue solid line with $+$ markers in the Figures \ref{fig1} and \ref{fig2})
\item $N=2^8$, $\epsilon=0.008$, $s=3$ (green dash line in the Figures \ref{fig1} and \ref{fig2})
\item $N=2^7$, $\epsilon=0.04$, $s=2$ (black dash-dot line in the Figures \ref{fig1} and \ref{fig2})
\end{itemize}
The results are shown in Figures \ref{fig1} and  \ref{fig2}.
\begin{figure}[htbp]
\centering
\includegraphics[scale=0.4]{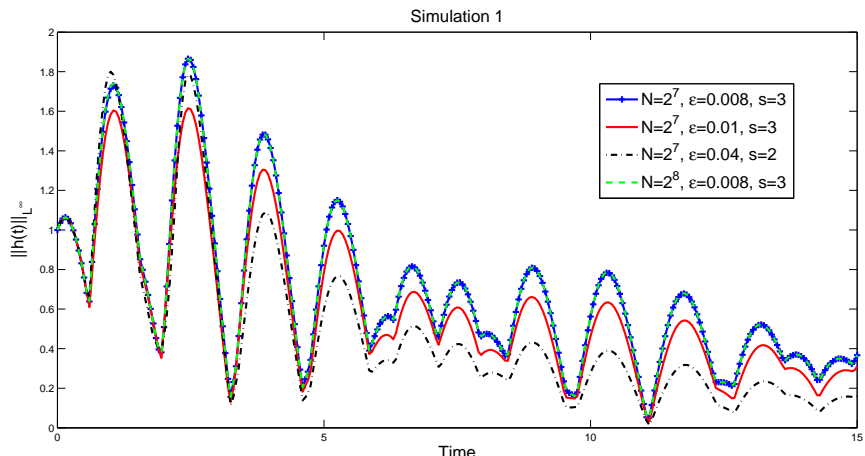}
\vspace{-.3 in}
\caption{\footnotesize{Evolution of the maximum amplitude of $|h(x,t)|$, written as $\|h(t)\|_{L^\infty}$.}}
\label{fig1}
\end{figure}

\begin{figure}[htbp]
\centering
\includegraphics[scale=0.3]{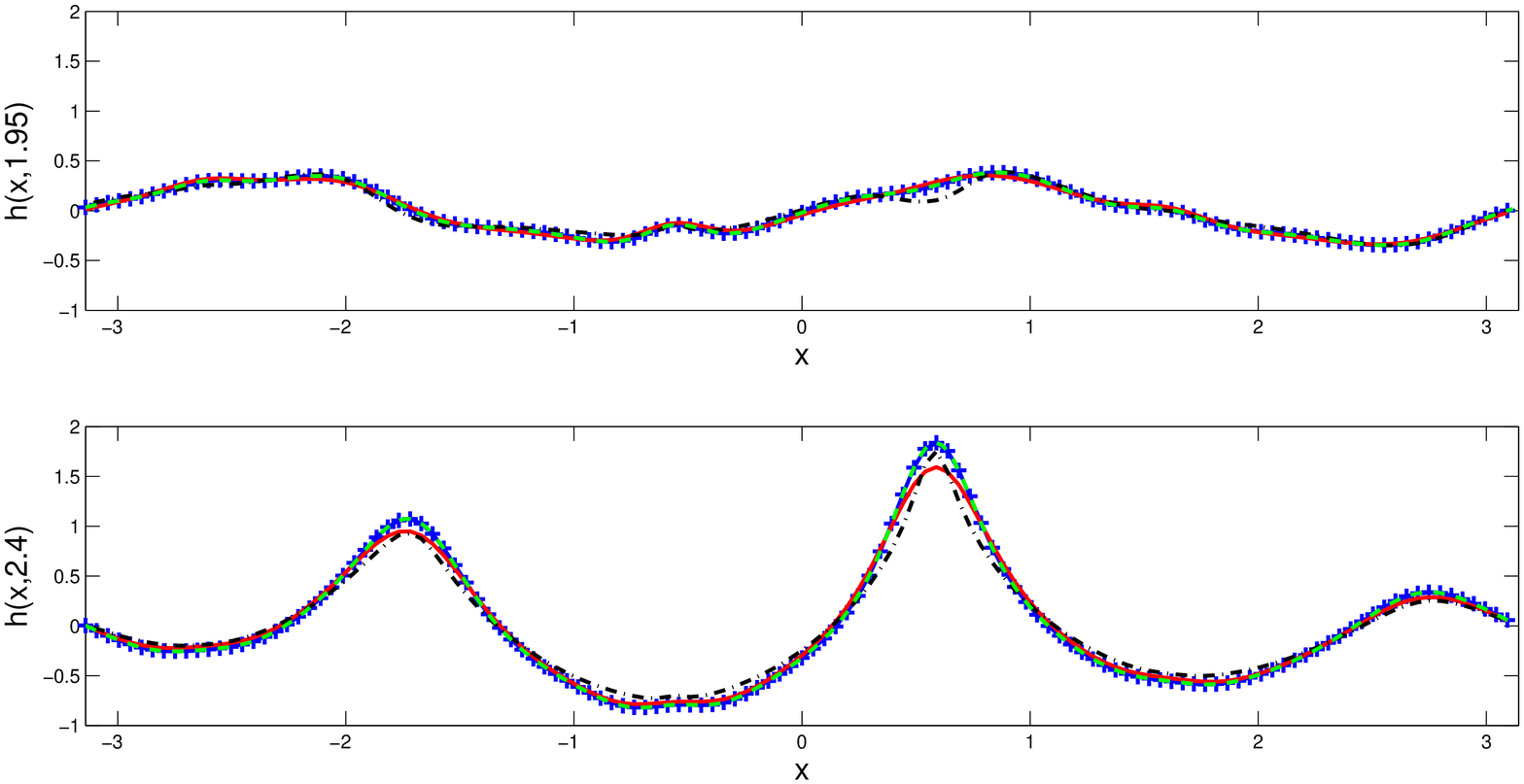}
\vspace{-.3 in}
\caption{{\footnotesize Interface position $h(x,t_j)$ for $t_1=1.95$ and $t_2=2.4$.}}
\label{fig2}
\end{figure}
%\begin{figure}
%\centering
%\includegraphics[scale=0.4]{EvolutionSim1b.eps}
%\caption{$h(x,t_j)$ for $t_1=1.95$ and $t_2=2.2$}
%\label{fig2b}
%\end{figure}
%\comentario{We have to chose just one figure}
Note that the results are qualitatively similar for various values of parameters, and, in particular, show the convergence of the numerical solutions with mesh size 
$N$. A comparison of the blue and the green curves in  Figures \ref{fig1} and \ref{fig2} demonstrates this nicely.

Furthermore, as shown in Figure \ref{fig1}, there is a large jump in the amplitude of the interface $\|h\|_{L^\infty}$ over a very small time scale. 
Specifically, we see that for an $O(1)$ initial interface position and velocity, the amplitude of the interface, $\|h\|_{L^\infty}$, grows
by a factor of $4$ in a time $0.45$.  
%Such interface growth appears to be independent of the Atwood number, and also  of $N,s$ and $\epsilon$.  
See Figure \ref{fig2} for the comparison of the interface postion $h(x,t_1)$ at time $t_1=1.95$ and interface position $h(x,t_2)$ at time $t_2=2.4$. 

%\textbf{The RT model with interface turn-over:} We start considering the initial data 
%\begin{align} 
%\delta z_1(\alpha,0)&=0\label{initialdatanumZ}\\
%z_2(\alpha,0)&=\sin(3\alpha)\label{initialdatanumZ2}\\
%\varpi(\alpha,0)&=2H\sin(2\alpha),\label{initialdatanum2Z}
%\end{align}
%This corresponds to the initial data \eqref{initialdatanum} and \eqref{initialdatanum2}. For $N=2^9$ and $\epsilon=0.08$, the results are plotted in figure \ref{fig8}
%
%\begin{figure}[htbp]
%\centering
%\includegraphics[scale=0.4]{sim1Z.eps}
%\vspace{-.3 in}
%\caption{\footnotesize{$z(\alpha,t_j)$ at times $t_1=0,t_2=0.4,t_3=0.85$ and $t_4=1$.}}
%\label{fig8}
%\end{figure}
%
%We observed that the system for $z$ seems more singular than the equation for $h$ in the sense that it requires higher viscosity to have comparable lifespans. Also, the solution develops large slopes before $t=0.9$, suggesting that turning may be occurring, \emph{i.e.} $z$ could not longer be parametrized as a graph (see the dash-dot curve in figure \ref{fig8}). 

\vspace{.2 in}

\subsection{Simulation 2: $h$-model,  $\frac{\rho^+}{\rho^-} = \frac{1.23}{1027}$ for Atwood number $A<0$}\label{sim2}
Next, we consider the physical parameters 
$$
g=9.8\,m/s,\;\sigma=0,\;\rho^+=1.23\, kg/m^3,\;\rho^-=1027\,kg/m^3,\text{ so }A=-0.99761.
$$
This corresponds to the density of air ($\rho^+$) and ocean water ($\rho^-$).

We consider the initial data \eqref{initialdatanum} and \eqref{initialdatanum2} for \eqref{nonlinearreg}. The artificial viscosity $\epsilon=0.05$ and the order of the artificial viscosity operator is $s=3$. The number of nodes is $N=2^7$.
\begin{figure}[htbp]
\centering
\includegraphics[scale=0.3]{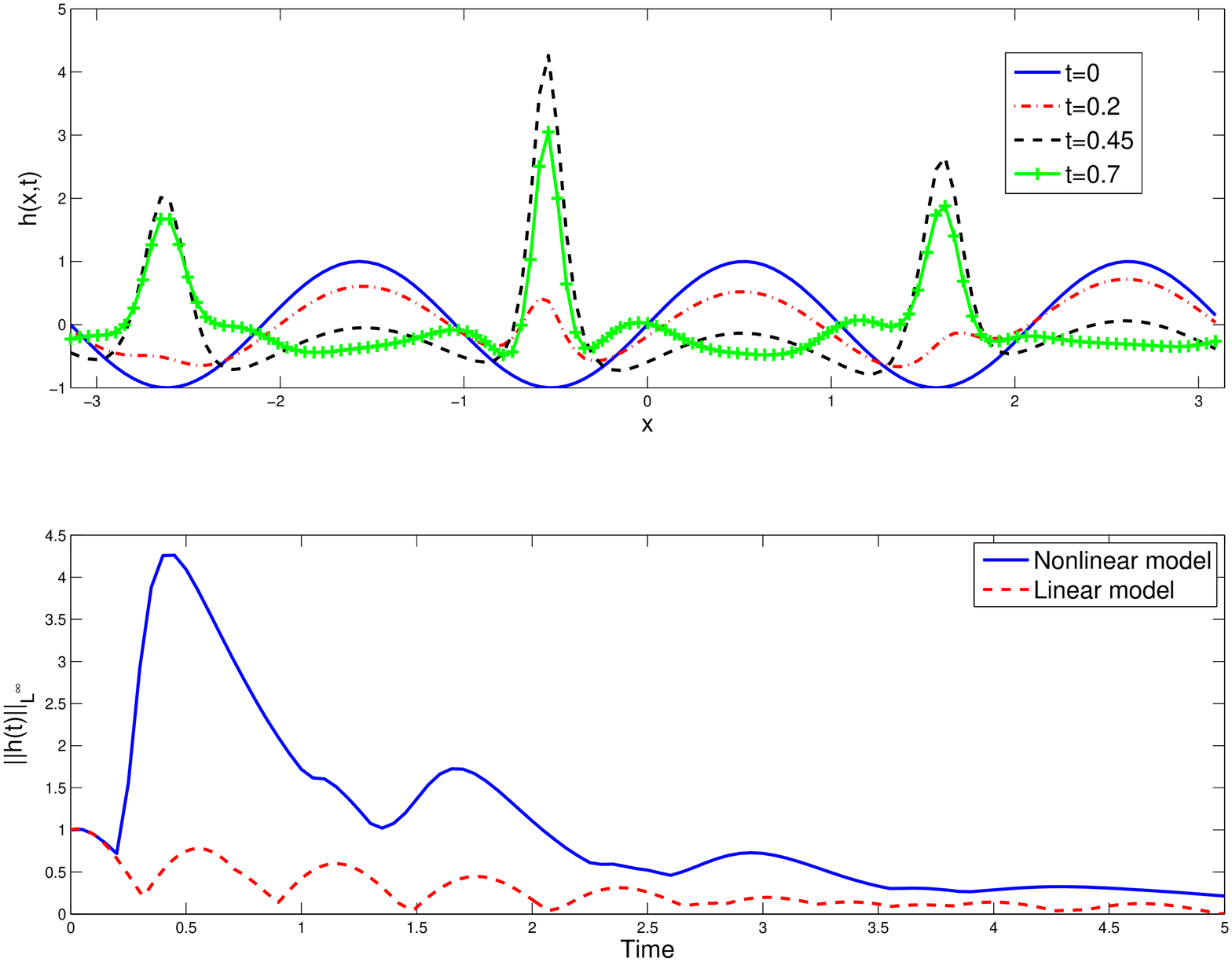}
\vspace{-.3 in}
\caption{{\footnotesize a) Interface position $h(x,t_j)$ for $t_1=0$, $t_2=0.2$, $t_3=0.45$ and $t_4=0.7$. b) $\|h(t)\|_{L^\infty}$ as a function of time.}}
\label{fig3}
\end{figure}
The results of this simulations are shown in Figure \ref{fig3}, wherein, we once again see the fast growth of the interface $\|h(t)\|_{L^\infty}$ over a short time-scale, followed, by decay to equilibrium.   The heavier fluid in this simulation, as compared to Simulation 1, reduces the frequency of oscillations, as
the interface decays to the rest state.   From the intial data,  the amplitude grows from  $0.7$ to $4.3$ in a time scale of length $0.2$ (between $t=0.2$ and $t=0.4$). After this remarkable growth, the amplitude of the interface decays  and approaches the rest state (see the blue curve in Figure \ref{fig3}).  In
order to demonstrate the role of the nonlinearity in the growth of the interface, we compare against the linear $h$-model \eqref{linear}.   As expected, the linear model simply decays the interface amplitude; see the red curve in Figure \ref{fig3}.

Finally, in Figure \ref{fig3spectrum} we plot the energy spectrum
$$
\mathcal{E}(k,t)=|\hat{h}_t(k,t)|^2-Ag|k||\hat{h}(k,t)|^2,
$$
associated to the energy law (\ref{energylaw}a) in Proposition \ref{prop2}, 
as a function of $k$ at $t_0=0, t_1=0.2, t_2=0.45$ and $t_3=0.7$.
\begin{figure}[htbp]
\centering
\includegraphics[scale=0.3]{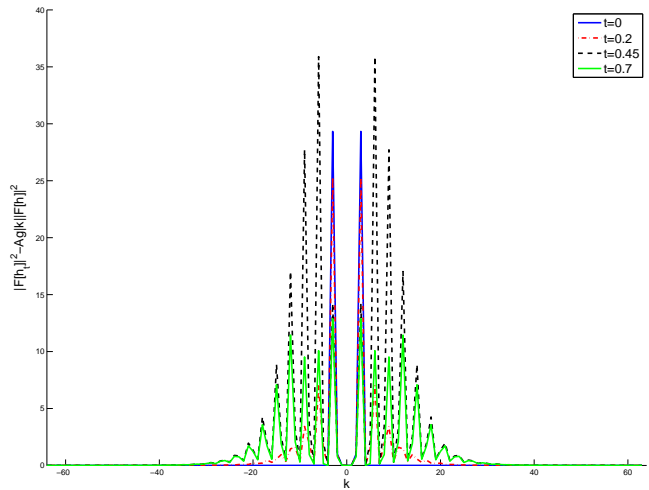}
\vspace{-.3 in}
\caption{{\footnotesize The energy spectrum $|\hat{h}_t(k,t)|^2-Ag|k||\hat{h}(k,t)|^2$ as a function of the Fourier mode $k$ at $t_0=0, t_1=0.2, t_2=0.45$ and $t_3=0.7$.}}
\label{fig3spectrum}
\end{figure}
We note that outside the Fourier modes $k\in[-50,50]$,  the energy spectrum $\mathcal{E}(k,t)$ is of order $10^{-4}$ for the time interval considered.  Starting from a mode-$3$ initial interface shape, the energy content is distributed into the smaller scales. At time $t=.45$ when the
interface is of maximum amplitude, the energy content is well-distributed among all large scales $|k|\le 20$.

%\textbf{The RT model with interface turn-over:} Next, we consider the initial data \eqref{initialdatanumZ}, \eqref{initialdatanumZ2} and \eqref{initialdatanum2Z}. The numerical parameters are $\epsilon=0.08$ and $N=2^9$. The results are plotted in figure \ref{fig9}.
%
%\begin{figure}[htbp]
%\centering
%\includegraphics[scale=0.4]{sim2Z.eps}
%\vspace{-.3 in}
%\caption{\footnotesize{$z(\alpha,t_j)$ at times $t_1=0,t_2=0.3,t_3=0.6$ and $t_4=0.9$.}}
%\label{fig9}
%\end{figure}
%
%Notice that the amplitude of the numerical solution $z$ behaves much more like the linear model even if solves a much nonlinear equation than $h$. In particular, the numerical solution to the simple RT model, $h$, described previously has larger growth of the amplitude. 

\vspace{.2 in}

\subsection{Simulation 3: $h$-model,  fingering instability, Atwood number $A>0$} \label{sim3} We next simulate the highly unstable case of a heavy fluid on top of
the lighter fluid, and with the acceleration acting downward. We consider the following physical parameters:
$$
g=9.8\,m/s,\;\sigma=0,\;\rho^+=10\, kg/m^3,\;\rho^-=1\,kg/m^3,\text{ so }A=0.81818 \,.
$$

We once again use our order one initial data \eqref{initialdatanum} and \eqref{initialdatanum2} for \eqref{nonlinearreg}. The artificial viscosity operator is order $s=3$, and the artificial viscosity parameter is set to $\epsilon=0.05$. The number of nodes is $N=2^7$.

This simulation is intended to demonstrate the ability of the $h$-model to show fingering phenomenon; 
indeed, as can be seen in Figure  \ref{fig4},  the heavy fluid penetrates the lighter fluid and a strong RT 
instability is initiated.   After a large enough time interval, the absolute value of the interface grows exponentially fast.
\begin{figure}[htbp]
\centering
\includegraphics[scale=0.35]{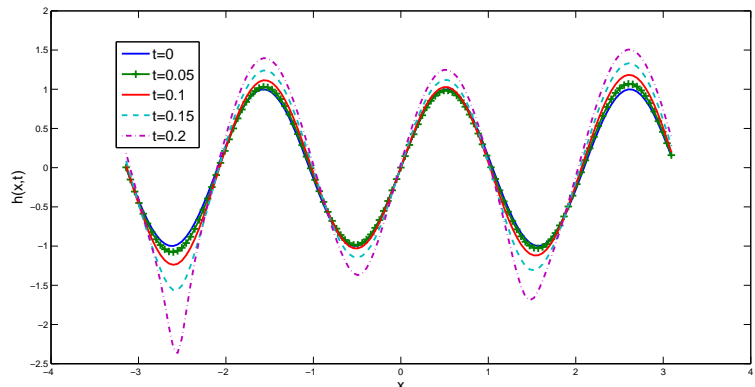}
\vspace{-.3 in}
\caption{{\footnotesize Interface position $h(x,t_j)$ for $t_1=0$, $t_2=0.05$, $t_3=0.1$, $t_4=0.15$ and $t_5=0.2$.}}
\label{fig4}
\end{figure}

%\textbf{The RT model with interface turn-over:} We consider the initial data given by \eqref{initialdatanumZ}, \eqref{initialdatanumZ2} and \eqref{initialdatanum2Z}. The numerical parameters are $\epsilon=0.01$ and $N=2^9$. The results are plotted in figure \ref{fig11}.
%
%\begin{figure}
%\centering
%\includegraphics[scale=0.4]{sim3Z.eps}
%\vspace{-.3 in}
%\caption{{\footnotesize Interface position $z(\alpha,t_j)$ for $t_1=0$, $t_2=0.2$, and $t_3=0.4$.}}
%\label{fig11}
%\end{figure}
%
%Our numerical solutions shows that the RT instability starts and after a short amount of time the interface can not be longer parametrized as a graph. However, the fact that now the solution is allowed to turn over impedes the cusp formation suggested by figure \ref{fig4}. 

\subsection{Simulation 4: $h$-model,  Stability, Atwood number $A<0$} \label{sim3b}
To capture the behavior described in Theorem \ref{thm2}, we simulate (\ref{wavenonlinearreg}') (instead of \eqref{nonlinearreg}). We consider the physical parameters
$$
g=9.8\,m/s,\;\sigma=0,\;\rho^+=0\, kg/m^3,\;\rho^-=1\,kg/m^3,\text{ so }A=-1 \,,
$$
and the initial data
$$
h_0=\frac{\cos(\alpha)}{10},\;h_1=-1+\frac{\sin(\alpha)}{10}.
$$

For this initial data, the homogeneous solution $h^\infty(t)$ 
is given by
$$
h^\infty(t)=-t.
$$
Theorem \ref{thm2} states that $h( \cdot , t) - h^ \infty (t)$ converges to zero as $ t \to\infty $.

As the initial data satisfies the stability condition,  no  artificial viscosity is required to stabilize the numerical solution, and we set $\epsilon=0$. Finally, we fix the number of Fourier modes  to $N=2^7$.
The results are plotted in Figure \ref{fig4b}, where the simulation demonstrates the asymptotic behavior of the theorem.
\begin{figure}[htbp]
\centering
\includegraphics[scale=0.3]{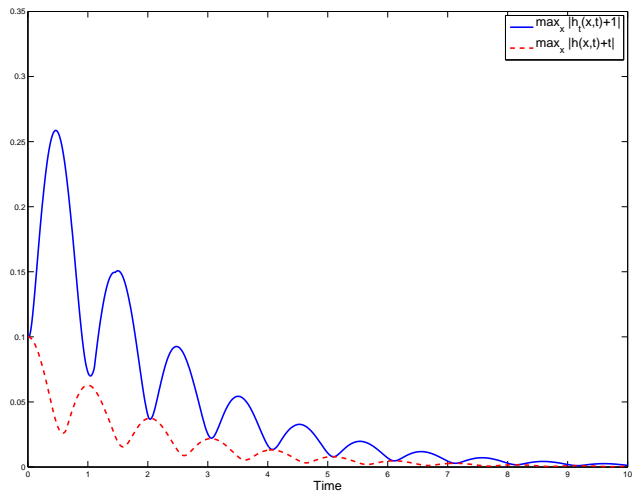}
\vspace{-.3 in}
\caption{{\footnotesize Distance from $h(x,t)$ to $h^\infty(t)$ (dashed line) and from $h_t(x,t)$ to $-1$}}
\label{fig4b}
\end{figure}

\subsection{Simulation 5: The ``rocket rig'' experiment of Read and Youngs}\label{sim4}
We consider now the situation where two (nearly) incompressible fluids, of densities $\rho^\pm$, are subjected to an approximately 
constant acceleration $g$ normal to the interface separating them, directed from the lighter fluid to the denser
fluid. We assume that the initial interface  is given by a \emph{small and random} perturbation of the flat state.

Our goal is to compare the growth rate of the mixing layer using our model equation \eqref{wavenonlinear} with that predicted by experiments and numerical simulations of 
Read \cite{Read1984} and Youngs \cite{Youngs1989}.
Experiments show that if the instability arises in the previous setting, the width of the mixed region grows like $t^2$. Actually, as shown
by Read \cite{Read1984} and Youngs \cite{Youngs1989}, the mixing region grows as
\begin{equation}\label{theoretical}
\delta Agt^2.
\end{equation}
Direct Numerical Simulation in two space dimensions by Youngs \cite{Youngs1989} has indicated that the parameter $\delta$ should range from 
0.04 to 0.05, whereas experiments of Read \cite{Read1984} suggest that  $\delta$ should range from 0.06 to 0.07.  In particular, when we have a NaI solution ($\rho^-=1.89g/cm^3$) 
and  Hexane ($\rho^+=0.66g/cm^3$) in a  tank, the empirical value is  $\delta=0.063$ (see \cite{Read1984}). Let us emphasize that both the
 numerical simulations and the physical experiments were run for approximately $70 ms$.

\subsubsection{The $h$-model} We consider the initial data for the $h$-model,  given by
\begin{align} 
h(\alpha,0)&=S\sum_{j=1}^{n}a_j\cos(jx)+b_j\sin(jx)\label{initialdatanumb}\\
\varpi(\alpha,0)&=0,\label{initialdatanum2b}
\end{align}
where $a_j,b_j$ are random numbers following a standard Gaussian distribution, and $S$ denotes a normalization constant such that
$$
\|h(0)\|_{L^2}=\frac{\pi}{100}.
$$

We consider $n=50$, $N=2^7$, and $\sigma=0$.

The acceleration
$g=-\frac{9.8\cdot2\cdot \pi}{0.3} L/s^2$ acts upwards, where  $L$  is chosen so that $2\pi/0.3 \, L=1 m$. This gravity force corresponds to the usual gravity force on the surface of the Earth of $9.8 m/s^2$. Finally, we use the artificial viscosity parameter $\epsilon=0.05$ together with a second-order artificial
viscosity operator $s=2$.    

The mixing layer is shown in Figure  \ref{fig6}, where the  interface position $h(x,t_j)$ is displayed for times $t_1=0$, $t_2=0.02$, $t_3=0.04$, $t_4=0.06$ and $t_5=0.08$.
\begin{figure}[htbp]
\centering
\includegraphics[scale=0.35]{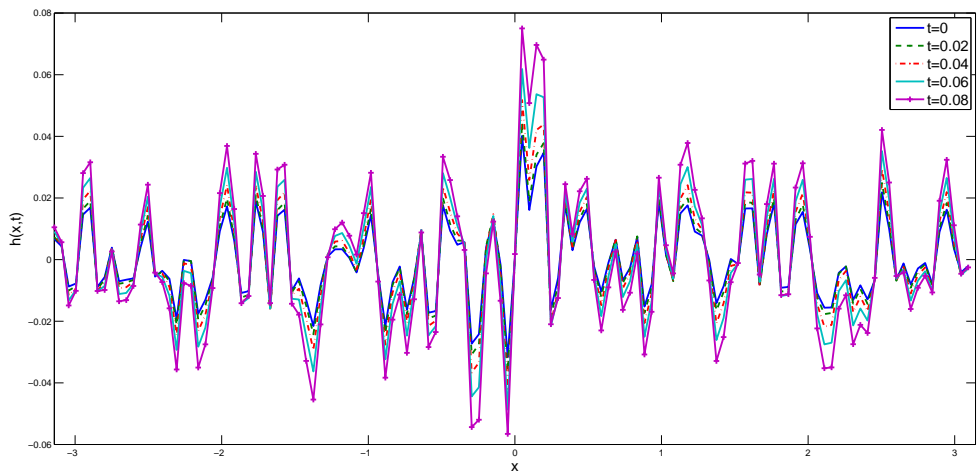}
\vspace{-.3 in}
\caption{{\footnotesize Interface position $h(x,t_j)$ for $t_1=0$, $t_2=0.02$, $t_3=0.04$, $t_4=0.06$ and $t_5=0.08$.}}
\label{fig6}
\end{figure}

In Figure \ref{fig5},
we see that up to time  $t=150 ms$,  the numerical solution provides a mixing-layer growth rate which agrees well
 with the predicted growth rate given by \eqref{theoretical}; the largest difference between the $h$-model and \eqref{theoretical} is given by
$$
\max_{0\leq t\leq 150 ms} \max_x h(x,t)-\max_x h(x,0)-0.06Ag t^2 =0.002916 \,.
$$
\begin{figure}[htbp]
\centering
\includegraphics[scale=0.25]{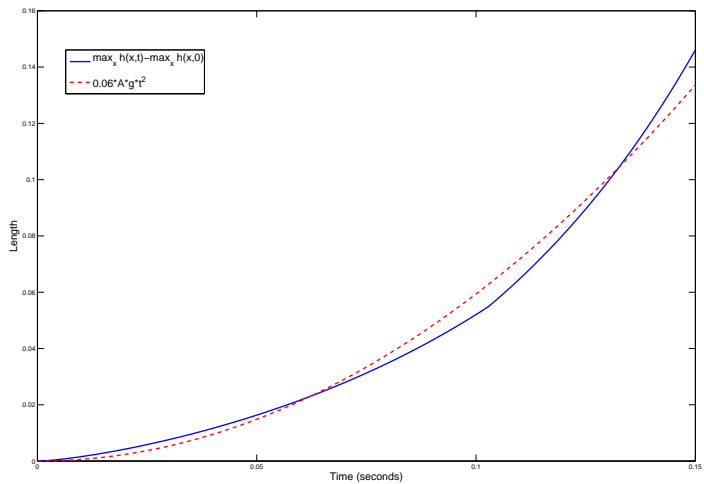}
\vspace{-.3 in}
\caption{{\footnotesize Comparison between $\max_xh(x,t)-\max_x h(x,0)$ and the predicted quadratic growth rate \eqref{theoretical} with $\delta=0.06$.}}
\label{fig5}
\end{figure}

For large times,  due to the strong RT instability present in the equations, the results of numerical simulations are very sensitive to the artificial 
viscosity parameter (see \cite{GlGrLiOhSh2001} for the artificial viscosity effects on RT mixing rates); however, for short time (meaning around $70ms$, neither viscosity nor nonlinearity plays a critical role in the evolution.  To
demonstrate this, we perform a numerical simulation of the $h$-model with zero artificial viscosity $\epsilon=0$, and with 
surface tension $\sigma=0.005$.   We use the  initial data satisfying \eqref{initialdatanumb}-\eqref{initialdatanum2b} with $n=30$, and $S$ chosen such that
$$
\|h(0)\|_{L^2}=\frac{\pi}{1000}.
$$
As can be seen in Figure \ref{fig7}, the growth rate predicted by our model agrees well with the 
Youngs' growth rate with $\delta=0.06$ up to around $t=60 ms$.   
\begin{figure}[htbp]
\centering
\includegraphics[scale=0.4]{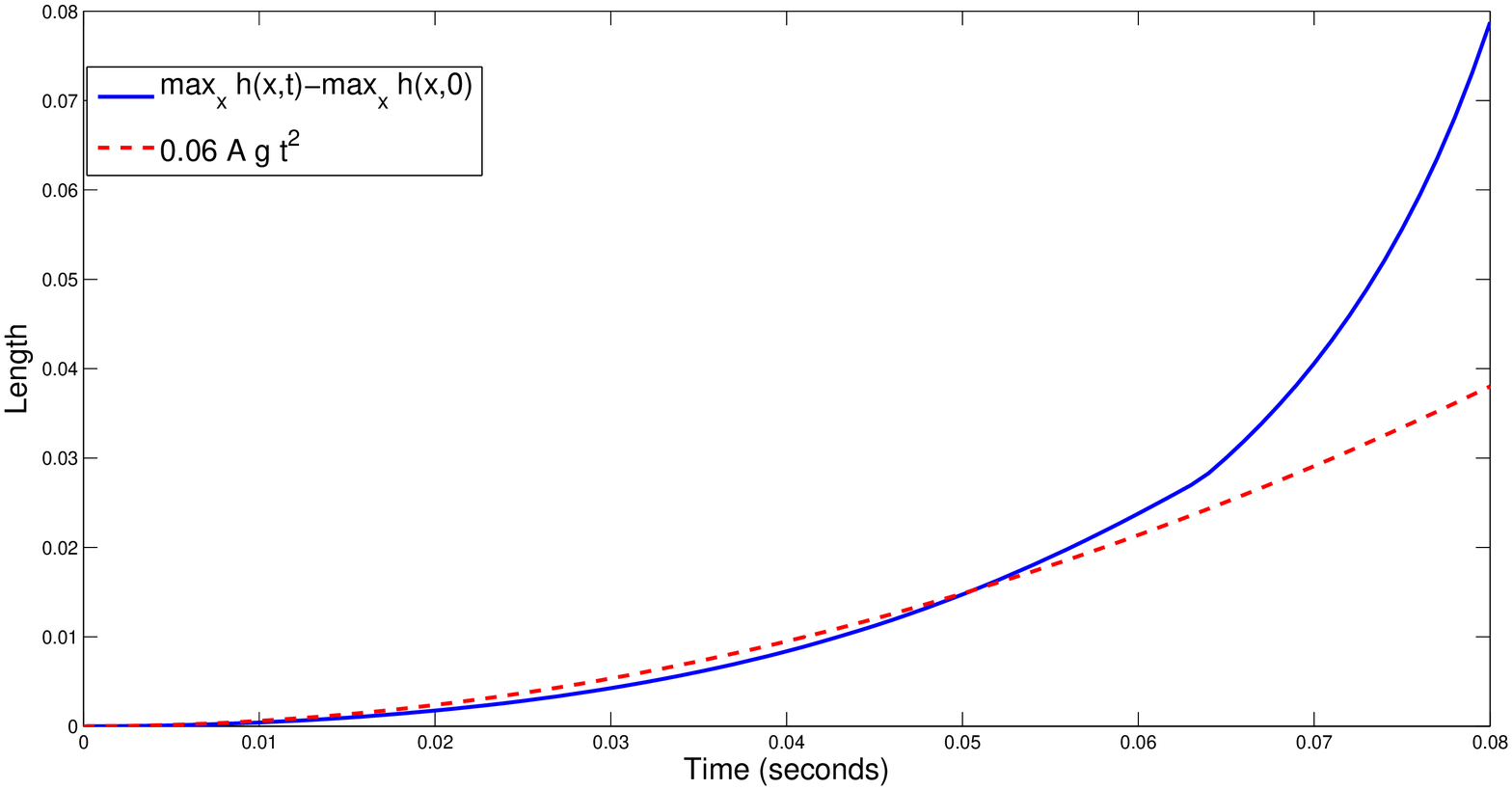}
\vspace{-.3 in}
\caption{{\footnotesize Comparison between $\max_xh(x,t)-\max_x h(x,0)$ and \eqref{theoretical} (with $\delta=0.06$) without artificial viscosity.}}
\label{fig7}
\end{figure}

\subsubsection{The $z$-model}
Now, we repeat the ``rocket rig'' experiment of Read and Youngs using the $z$-model, with initial data 
%We also compare the behavior of $h$ and $z$.
 given by
\begin{align} 
\delta z_1(\alpha,0)&=0\label{initialdatanumbZ}\\
z_2(\alpha,0)&=S\sum_{j=1}^{n}a_j\cos(jx)+b_j\sin(jx)\label{initialdatanumbZ2}\\
\varpi(\alpha,0)&=0,\label{initialdatanum2bZ}
\end{align}
where $a_j,b_j$ are random numbers following a standard Gaussian distribution, and $S$ denotes a normalization constant such that
$$
\|z_2(0)\|_{L^2}=\frac{\pi}{1000}.
$$
We consider $n=30$ and $N=2^9$. The artificial viscosity for \eqref{nonlinearregZ} has coefficient $\epsilon=0.01$.   Figure \ref{fig12b} shows
the interface evolution for the $z$-model.
\begin{figure}[htb]
\centering
\includegraphics[scale=0.3]{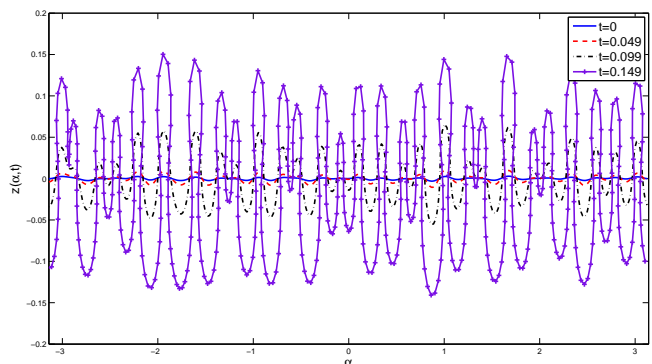}
\vspace{-.3 in}
\caption{{\footnotesize Interface position $z(\alpha,t_j)$ for $t_0=0$, $t_1=0.049$, $t_2=0.099$ and $t_3=0.149$.}}
\label{fig12b}
\end{figure}
The ability of the $z$-model parameterization to ``fatten'' and ``finger'' produces an even more accurate representation of the mixing region.

We are able to quantitatively validate this statement by comparing the growth of the mixing region of the $z$-model with the $h$-model and
the quadratic prediction \eqref{theoretical} of Read and Youngs.
For the comparison, we simulate the $h$-model with the same physical parameters,  and with second-order artificial viscosity 
with $\epsilon=0.05$, and with initial data given by
\begin{align*} 
h(\alpha,0)&=z_2\\
\varpi(\alpha,0)&=0.
\end{align*}
We note  that the artificial viscosity $ \epsilon $  is five times larger for the $h$-model  \eqref{nonlinearreg} than for the $z$-model
 \eqref{nonlinearregZ}.
As shown in Figure \ref{fig12}, the $z$-model mixing region growth rate matches very well with the quadratic prediction \eqref{theoretical}
of Read and Youngs, and for a relatively long time interval, up to $t= 150ms$.
\begin{figure}[htb]
\centering
\includegraphics[scale=0.35]{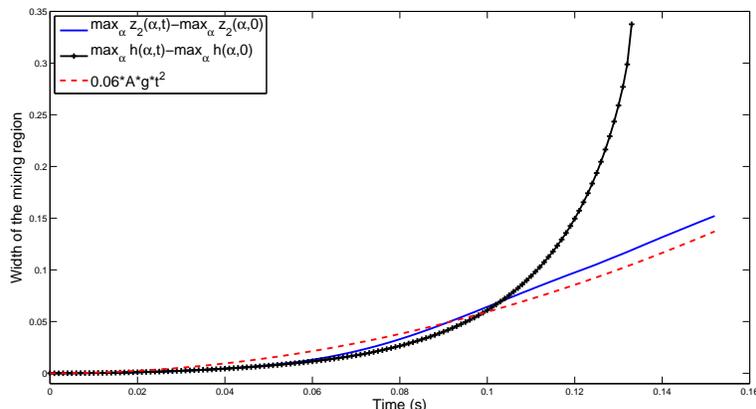}
\vspace{-.3 in}
\caption{{\footnotesize Comparison between $\max_xz_2(x,t)-\max_x z_2(x,0)$, $\max_x h(x,t)-\max_x h(x,0)$ and \eqref{theoretical} with $\delta=0.06$.}}
\label{fig12}
\end{figure}
The width of the mixing region,
approximated by
$$
\max_\alpha z_2(\alpha,t)-\max_\alpha z_2(\alpha, 0) \,,
$$
is in excellent agreement with the quadratic prediction
\eqref{theoretical}.  The maximum error for the $z$-model is
$$
\max_{0\leq t \leq 0.152}\left|\max_\alpha z_2(\alpha,t)-\max_\alpha z_2(\alpha, 0)-0.06Ag t^2\right|\approx 0.01537 \,,
$$
whereas the maximum error for the $h$-model is
$$
\max_{0\leq t \leq 0.133}\left|\max_\alpha h(\alpha,t)-\max_\alpha h(\alpha, 0)-0.06Ag t^2\right|\approx 0.23244 \,.
$$

Moreover, 
the $z$-model has a longer lifespan than the $h$-model.  Since the $h$-model is constrained to remain a graph, the amplitude of
the interface can only
grow rapidly once the RT instability is strongly initiated; on the other hand, as the $z$-model can turn-over, the instability creates a turn-over
and horizontal-fattening of the mixing region.  The comparison is shown Figure \ref{fig13}
\begin{figure}[htb]
\centering
\includegraphics[scale=0.3]{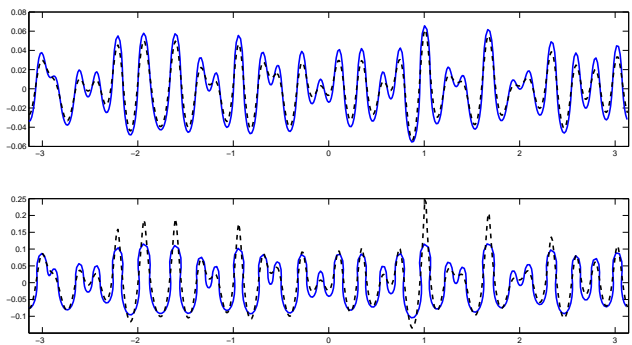}
\vspace{-.3 in}
\caption{{\footnotesize Comparison between $h(\alpha,t_j)$ (dashed line) and $z(\alpha,t_j)$ (solid line) at $t_1=0.099$ and $t_2=0.129$.}}
\label{fig13}
\end{figure}

\subsection{Simulation 6: The ``tilted rig'' experiment}\label{sim5}
Finally, we consider the ``tilted'' rig experiment of Youngs \cite{Youngs1989}, wherein the tank of  the rocket rig experiment is
titled by a small angle $\theta$ from the vertical.

Again, we have 
two (nearly) incompressible fluids, of densities $\rho^\pm$, with a vertical acceleration $g$ 
directed from the lighter fluid to the denser fluid.
As Youngs notes, 
the inclination of the initial interface results
in a gross overturning motion in addition to the
fine scale mixing.

We define $\theta=5.7^{\circ}$ and consider the unperturbed (flat) \emph{tilted} interface given by
$$
\tilde z^2_0=\left\{\begin{array}{cc}
    \tan(\theta)(x+\pi) & \text{ if } -\pi\leq x<-\pi/2,\\
   -\tan(\theta)x & \text{ if }|x|\leq\pi/2,\\
    \tan(\theta)(x-\pi)& \text{ if }\pi/2<\pi\leq \pi.
\end{array}\right.
$$
We assume that the actual initial interface is given by a \emph{small and random} perturbation so that
$$
z^2_0(x)=\tilde z^2_0(x)+S\sum_{j=1}^{n}a_j\cos(jx)+b_j\sin(jx),
$$
where $a_j,b_j$ are random numbers following a standard Gaussian distribution, and $S$ denotes a normalization constant such that
$$
\|z^2_0- \tilde z^2_0\|_{L^2}=\frac{\pi}{1000}.
$$
We fix $n=30$ and $N=2^9$.

\subsubsection{The $h$-model}
The initial data for the  $h$-model \eqref{nonlinearreg} is given by
\begin{align} 
h(\alpha,0)&=z^2_0(\alpha)\label{initialdatanumb2}\\
\varpi(\alpha,0)&=0,\label{initialdatanum2b2} \,.
\end{align}
Again, we employ a second-order artificial viscosity operator with $ \epsilon =0.25$, a rather large artificial viscosity, but necessary to 
stabilize the interface in the highly unstable RT regime.
The results are plotted in figure \ref{fig14}. 
\begin{figure}[htb]
\centering
\includegraphics[scale=0.3]{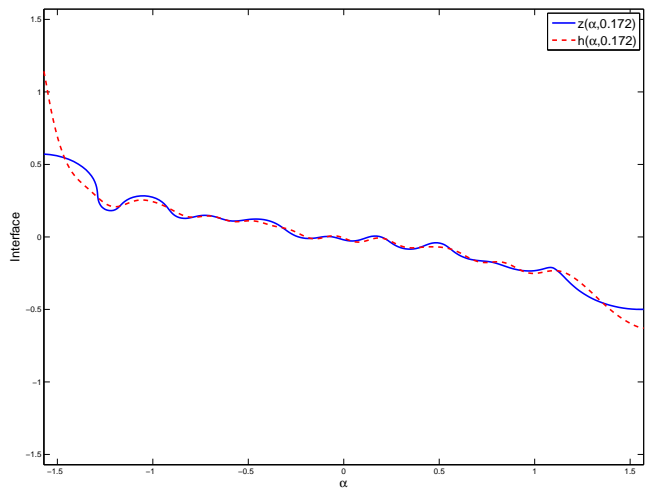}
\vspace{-.3 in}
\caption{{\footnotesize Interface position $h(\alpha,t_j)$ (dashed line) and $z(\alpha,t_j)$ (solid line) for $t=0.172$.}}
\label{fig14}
\end{figure}

\subsubsection{The $z$-model}
For the $z$-model, we use the 
 initial conditions
%\begin{align} 
%\delta z_1(\alpha,0)&=0\label{initialdatanumbZ2b}\\
%z_2(\alpha,0)&=z^2_0(x)\label{initialdatanumbZ22b}\\
%\varpi(\alpha,0)&=0,\label{initialdatanum2bZ2b}
%\end{align}
$$
\delta z_1(\alpha,0)=0\,, \ 
z_2(\alpha,0)=z^2_0(\alpha)\,, \ 
\varpi(\alpha,0)=0 \,,
$$
with artificial viscosity $\epsilon=0.05$.   As can be seen in Figure \ref{fig14}, as the RT instability is strongly initiated, the amplitude of the
$h$-model starts to grow unboundedly, while the $z$-model can turn-over and continue to run for a much longer time interval.

We show the time evolution of the interface for the $h$-model and $z$-model in Figure  \ref{fig15}.
\begin{figure}[htbp]
\subfigure[$h$-model]{
\label{fig15}
\includegraphics[scale=0.22]{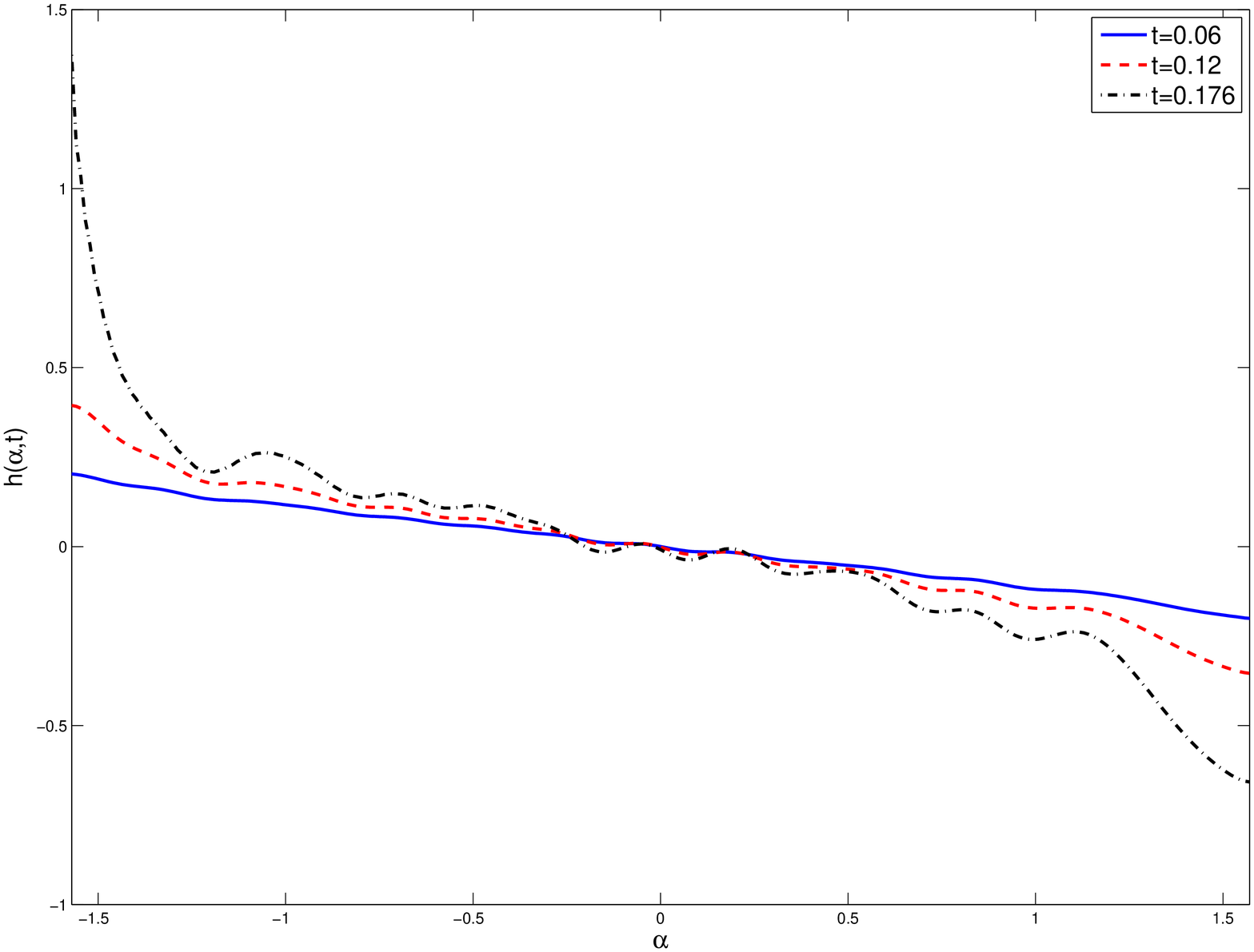}
}
\subfigure[$z$-model]{
\label{fig15a}
\includegraphics[scale=0.22]{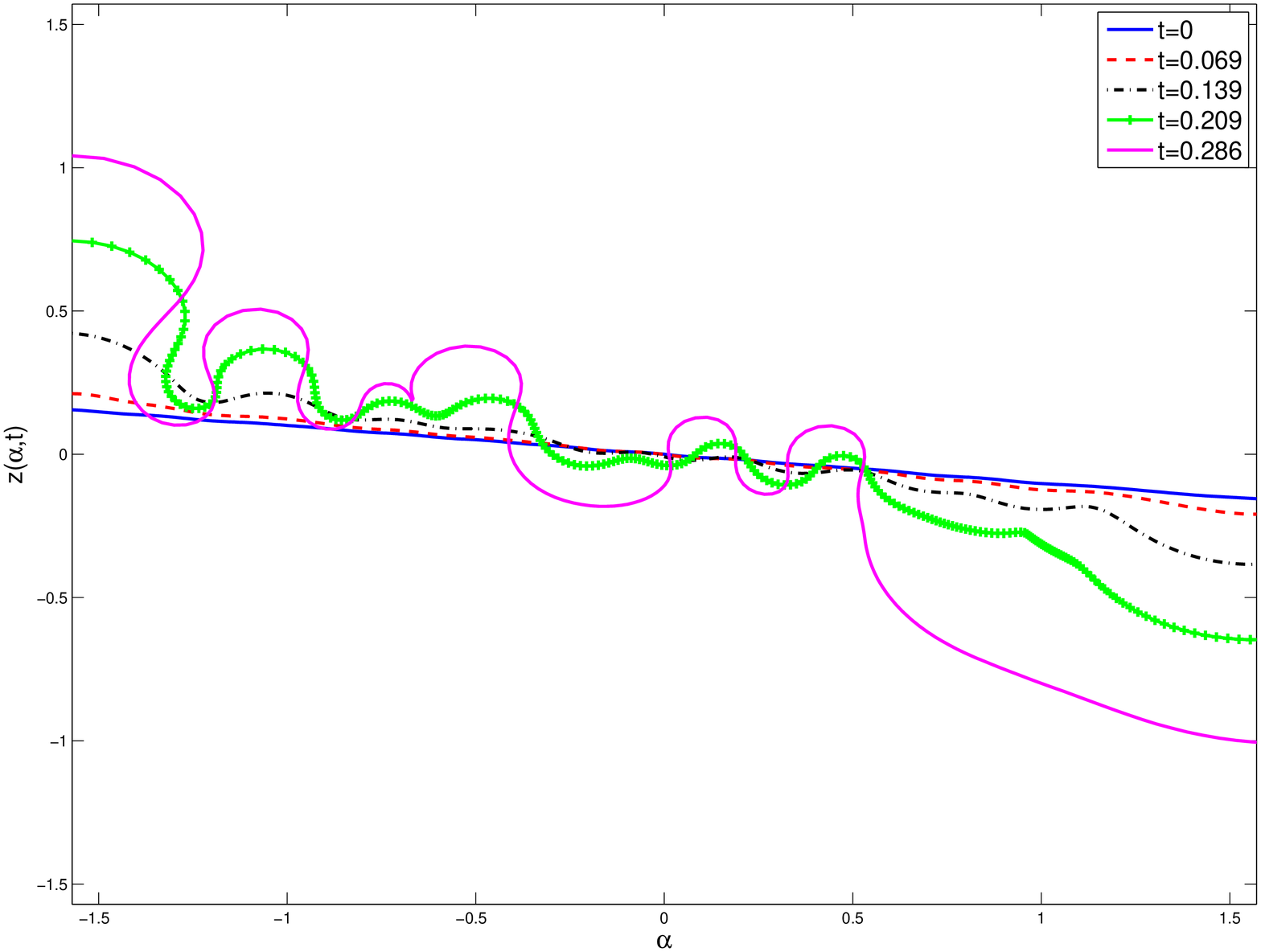}
}
\caption{{\footnotesize Interface position  for $t_0=0$, $t_1=0.069,t_2=0.139,t_3=0.209$ and $t_4=0.286$.}}

\label{fig15b}
\end{figure}
Because the $h$-model can only grow the amplitude of $h$, the behavior of the interface is more singular for the $h$-model  than for
the $z$-model. In particular, $\|h(t)\|_{L^\infty}$ grows faster than $\|z_2(t)\|_{L^\infty}$.   Moreover, the $z$-model is more stable, and permits
the use of smaller
artificial viscosity than the $h$-model.

Finally, as we have run the titled rig experiment with random initial data, it is interesting to plot the ensemble of runs, and get a clear picture of
the mixing region. In particular, as it  may be difficult to infer the qualitative description of the mixing region from only one simulation, in
Figure \ref{fig16},  we plot the  $z$-model interface location for a number of different random initial conditions at time $t=0.22$.
\begin{figure}[htbp]
\centering
\includegraphics[scale=0.3]{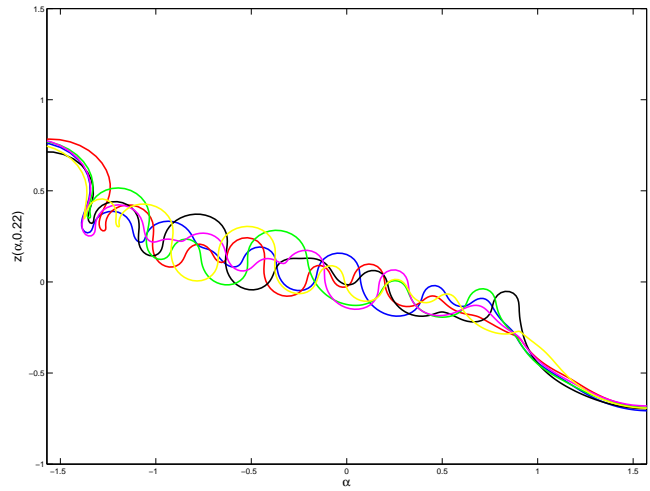}
\vspace{-.3 in}
\caption{{\footnotesize An ensemble of $z$-model interface positions for   $t=0.22$.}}
\label{fig16}
\end{figure}
As can be seen, the ensemble gives an approximation of the mixing region provided by experiment \cite{Youngs1989} and  DNS \cite{RoAn2013}.
The mixing region grows slightly faster on the left side than it drops on the right side as expected by Youngs \cite{Youngs1989}, and
has the qualitative features of the experiment.

\subsection{Simulation 7: $z$-model, Kelvin-Helmholtz instability, Atwood number $A=0$}
The $z$-model can simulate the Kelvin-Helmholtz instability arising in the case of equal densities $\rho^+=\rho^-$, equivalently $A=0$. We note that when $A=0,$ due to \eqref{z-dynamics3}, we have
$$
\varpi_t=0,
$$
and the problem reduces to describing the interface evolution for $z( \alpha , t)$  via (\ref{nonlinearregZ}a,b). We consider the initial data given by
$$
\delta z_1(\alpha,0)=-\sin(\alpha)\,, \ 
z_2(\alpha,0)=0.5\sin(\alpha)\,, 
$$
and
$$
\varpi(\alpha)=10\cos(\alpha) \,,
$$
and use $N=2^{8}$ Fourier modes.
We fix the artificial viscosity as $\epsilon=0.01$. The results are given in Figure \ref{fig17}.   The primary  effect of the nonlinearity is to increase 
the length of the interface by starting to roll-over, while keeping the curve smooth. 
\begin{figure}[htbp]
\centering
\includegraphics[scale=0.25]{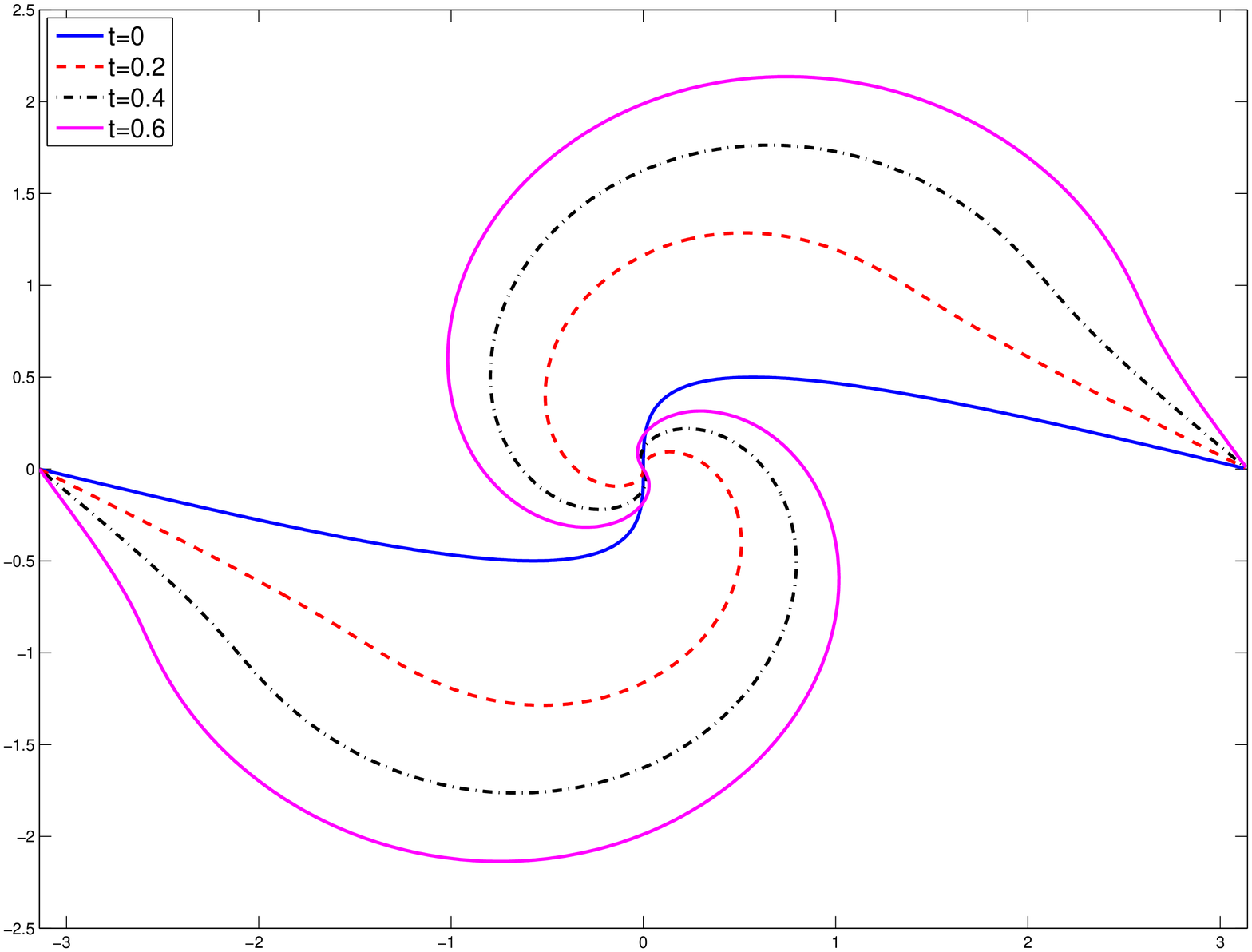}
\vspace{-.3 in}
\caption{{\footnotesize The interface $z(\alpha,t)$ at $t_0=0, t_1=0.2, t_2=0.4$ and $t_3=0.6$.}}
\label{fig17}
\end{figure}
\begin{figure}[htbp]
\centering
\includegraphics[scale=0.25]{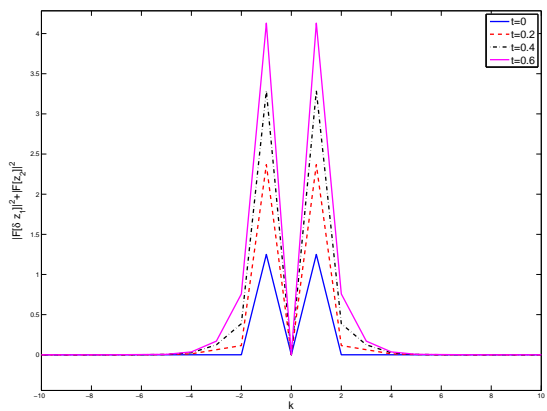}
\vspace{-.3 in}
\caption{{\footnotesize The energy spectrum $|\delta \hat{z}_1(k,t)|^2+|\hat{z}_2(k,t)|^2$ as a function of the Fourier mode $k\in [-10,10]$ at $t_0=0, t_1=0.2, t_2=0.4$ and $t_3=0.6$.}}
\label{fig18}
\end{figure}

This can be seen by looking at the spectral content of the  the $L^2$-energy function
$$
\mathcal{E}(k,t)=|\delta \hat{z}_1(k,t)|^2+|\hat{z}_2(k,t)|^2
$$
at  different instances of time.    
In Figure \ref{fig18} we plot the spectrum of $ \mathcal{E}(k,t)$ 
at times $t_0=0, t_1=0.2, t_2=0.4$ and $t_3=0.6$.   For large Fourier modes $|k| > 10$, the 
$L^2$-energy spectrum $\mathcal{E}(k,t)$ is of order $10^{-5}$ for the time interval considered, so we plot $\mathcal{E}(k,t)$ versus
$-10 \le k \le 10$.    The energy spectrum remains fairly localized about the $k=1$ initial data, with some growth in wavenumbers $|k|=2$
and $3$, which are responsible for the dilation and rotation of the wave.

\vspace{.1in}

The MATLAB code is available at 
\url{https://www.math.ucdavis.edu/~shkoller/RTcode}.

\vspace{.1in}

\noindent {\bf Acknowledgments.}  SS was supported by the National Science Foundation under grant  DMS-1301380,
and by the Department of Energy Advanced Simulation and Computing (ASC) Program.


\begin{thebibliography}{10}



\bibitem{CoCoGa2010}
A.~C{\'o}rdoba,  D.~C{\'o}rdoba, and F.~Gancedo.
\newblock{Interface evolution: water waves in 2-D}.
\newblock {\em {Advances in Mathematics}}, {223}({1}):{120--173}, {2010}.

\bibitem{GlGrLiOhSh2001}
J.~Glimm, J.W. Grove, X.L. Li, W.~Oh, and D.H. Sharp.
\newblock {A critical analysis of Rayleigh-Taylor growth rates}.
\newblock {\em {Journal of Computational Physics}}, {169}({2}):{652--677}, {May
  20} {2001}.

\bibitem{Goncharov2002}
V.N. Goncharov.
\newblock {Analytical model of nonlinear, single-mode, classical
  Rayleigh-Taylor instability at arbitrary Atwood numbers}.
\newblock {\em {Physical Review Letters}}, {88}({13}), {Apr 1} {2002}.

\bibitem{granero2014global}
R. Granero-Belinch{\'o}n.
\newblock Global existence for the confined muskat problem.
\newblock {\em SIAM Journal on Mathematical Analysis}, 46(2):1651--1680, 2014.

\bibitem{kenig1991well}
C.E.~Kenig, G.~Ponce, and L.~Vega.
\newblock Well-posedness of the initial value problem for the korteweg-de vries
  equation.
\newblock {\em Journal of the American Mathematical Society}, 4(2):323--347,
  1991.

\bibitem{Kull1991}
H.J. Kull.
\newblock {Theory of the Rayleigh-Taylor instability}.
\newblock {\em {Physics Reports-Review Section of Physics Letters}},
  {206}({5}):{197--325}, {Aug} {1991}.

\bibitem{Ra1878}
L.~Rayleigh.
\newblock On the instability of jets.
\newblock {\em Proceedings of the London Mathematical Society}, s1-10:4--13,
  1878.

\bibitem{Read1984}
K.I. Read.
\newblock Experimental investigation of turbulent mixing by rayleigh-taylor
  instability.
\newblock {\em Physica D}, 12(1):45--58, 1984.

\bibitem{ReSeSh2013}
J.~Reisner, J.~Serencsa, and S.~Shkoller.
\newblock {A space-time smooth artificial viscosity method for nonlinear
  conservation laws}.
\newblock {\em {Journal of Computational Physics}}, {235}:{912--933}, {Feb 15}
  {2013}.

\bibitem{RoAn2013}
B.~Rollin and M.~J. Andrews.
\newblock {On generating initial conditions for turbulence models: the case of
  Rayleigh-Taylor instability turbulent mixing}.
\newblock {\em {Journal of Turbulence}}, {14}({3}):{77--106}, {Mar 1} {2013}.

\bibitem{Sharp1984}
D.H. Sharp.
\newblock {An overview of Rayleigh-Taylor instability}.
\newblock {\em {Physica D}}, {12}({1-3}):{3--18}, {1984}.

\bibitem{Ta1950}
G.~Taylor.
\newblock The instability of liquid surfaces when accelerated in a direction
  perpendicular to their planes. 1.
\newblock {\em Proc. Roy. Soc. London. Ser. A.}, 201:192--196, 1950.

\bibitem{Youngs1984}
D.L. Youngs.
\newblock {Numerical-simularion of turbulent mixing by Rayleigh-Taylor
  instability}.
\newblock {\em {Physica D}}, {12}({1-3}):{32--44}, {1984}.

\bibitem{Youngs1989}
D.L. Youngs.
\newblock {Modeling turbulent mixing by Rayleigh-Taylor instability}.
\newblock {\em {Physica D}}, {37}({1-3}):{270--287}, {Jul} {1989}.

\end{thebibliography}
\end{document}